\documentclass[11pt]{article}
	\usepackage{caption}
    \usepackage{url}
    \usepackage{verbatim}
	\usepackage{amsmath}
	\usepackage{amsthm}
	\usepackage{amsfonts}
	\usepackage{graphicx}
	\usepackage{nicefrac}
	\usepackage{color}
	\usepackage{xcolor}
	\usepackage{soul}
	\usepackage{calc}
	\usepackage{mathtools}

\newlength\dlf

	\def\captionfont{\setb@se{11pt}\protect\footnotesize}
    \def\captionfont{\protect\footnotesize}

    \newcommand{\iprd}[2]{\left( #1 , #2 \right)}
    
    \newcommand{\Tauh}{\mathcal{T}_h}
    
    \newcommand{\aiprd}[2]{a\left( #1 , #2 \right)}
    \newcommand{\bform}[3]{b\left( #1 , #2, #3 \right)}
    \newcommand{\Bform}[3]{B\left( #1 , #2, #3 \right)}
    \newcommand{\cform}[2]{c\left( #1 , #2 \right)}

    \newcommand{\Soh}{\mathring{S}_h}
    \newcommand{\Xh}{{\bf X}_h}
    \newcommand{\Vh}{{\bf V}_h}

	\def\norm#1#2{\left\| #1 \right\|_{#2}}

	\newcommand{\dtau}{\delta_\tau}
	\newcommand{\ddtau}{\delta_{\tau}^2}

	\newcommand{\phih}{\phi_h}
	
	\newcommand{\phitil}{\tilde{\phi}_h}
	\newcommand{\phich}{\check{\phi}_h}
	\newcommand{\bphih}{\bar{\phi}_h}
	\newcommand{\bbuh}{\bar{\bf u}_h}
	\newcommand{\bu}{{\bf u}}
	\newcommand{\bv}{{\bf v}}
	\newcommand{\bw}{{\bf w}}
	\newcommand{\buh}{{\bf u}_h}
	
	\newcommand{\butil}{\tilde{\bu}_h}
	
	\newcommand{\muh}{\mu_h}

	\newcommand{\hf}{\frac{1}{2}}
	\newcommand{\eA}{\mathcal{E}_a}
	\newcommand{\eh}{\mathcal{E}_h}
	\newcommand{\e}{\mathcal{E}}
	\newcommand{\eAphi}{\mathcal{E}_a^{\phi}}
	\newcommand{\ehphi}{\mathcal{E}_h^{\phi}}
	\newcommand{\ephi}{\mathcal{E}^{\phi}}

	\newtheorem{thm}{Theorem}[section]
	\newtheorem{prop}[thm]{Proposition}

	\newtheorem{lem}[thm]{Lemma}
	\newtheorem{rem}[thm]{Remark}
	\newtheorem{asmp}[thm]{Assumption}

\begin{document}
\title{Convergence Analysis and Error Estimates for a Second Order Accurate Finite Element Method for the Cahn-Hilliard-Navier-Stokes System}

	\author{
Amanda E. Diegel\thanks{Department of Mathematics, Louisiana State University, Baton Rouge, LA 70803 (adiegel@lsu.edu)} 
	\and
Cheng Wang\thanks{Department of Mathematics, The University of Massachusetts, North Dartmouth, MA 02747 (cwang1@umassd.edu)} 
	\and
Xiaoming Wang\thanks{Department of Mathematics, Florida State University, Tallahassee, FL 32306 (wxm@math.fsu.edu)}
	\and
Steven M. Wise\thanks{Department of Mathematics, The University of Tennessee, Knoxville, TN 37996 (swise@math.utk.edu)}}

	\maketitle
	
	\numberwithin{equation}{section}
	
\begin{abstract}
In this paper, we present a novel second order in time mixed finite element scheme for the Cahn-Hilliard-Navier-Stokes equations with matched densities. The scheme combines a standard second order Crank-Nicholson method for the Navier-Stokes equations and a modification to the Crank-Nicholson method for the Cahn-Hilliard equation. In particular, a second order Adams-Bashforth extrapolation and a trapezoidal rule are included to help preserve the energy stability natural to the Cahn-Hilliard equation. We show that our scheme is unconditionally energy stable with respect to a modification of the continuous free energy of the PDE system. Specifically, the discrete phase variable is shown to be bounded in $\ell^\infty \left(0,T;L^\infty\right)$ and the discrete chemical potential bounded in $\ell^\infty \left(0,T;L^2\right)$, for any time and space step sizes, in two and three dimensions, and for any finite final time $T$. We subsequently prove that these variables along with the fluid velocity converge with optimal rates in the appropriate energy norms in both two and three dimensions.
	\end{abstract}
	
{\bf Keywords:} Cahn-Hilliard equation, Navier-Stokes, mixed finite element methods, convex splitting, energy stability, error estimates, second order

\medskip

{\bf AMS subject classifications.} 35K35 35K55 65M12 65M60

	\section{Introduction}
In this paper, we prove error estimates for a fully discrete, second order in time, finite element method for the Cahn-Hilliard-Navier-Stokes (\emph{CHNS}) model for two-phase flow. Let $\Omega\subset \mathbb{R}^d$, $d=2,3$, be an open polygonal or polyhedral domain. For all $\phi \in H^1(\Omega), \bu \in {\bf L}^2(\Omega)$, consider the energy 
	\begin{align}
E(\phi, \bu) = \int_{\Omega} \left\{\frac{1}{4\varepsilon} \left(\phi^2 - 1\right)^2 + \frac{\varepsilon}{2} |\nabla \phi|^2 + \frac{1}{2\gamma} |\bu |^2 \right\} d\bf{x},
	\label{eq:Ginzburg-Landau-energy-modified}
	\end{align}
where $\phi$ represents a concentration field, $\bu$ represents fluid velocity, and $\varepsilon$ is a positive constant. The \emph{CHNS} system is a gradient flow of this energy~\cite{gurtin96, han15, hohenberg77, kay08}:
	\begin{subequations}
	\begin{align}
\partial_t \phi + \nabla \phi \cdot \bu = \varepsilon \nabla \cdot \left(M(\phi) \nabla \mu\right), &\quad \text{in} \,\Omega_T ,
	\label{eq:chns-mixed-a}
	\\
\mu = \varepsilon^{-1}\left(\phi^3-\phi\right)-\varepsilon \Delta \phi, &\quad \text{in} \,\Omega_T,   
	\label{eq:chns-mixed-b}
	\\
\partial_t \bu - \eta \Delta \bu + \bu \cdot \nabla \bu + \nabla p =  \gamma \mu \nabla \phi, &\quad \text{in} \, \Omega_T,
	\label{eq:chns-mixed-c}
	\\
\nabla \cdot \bu = 0, & \quad \text{in} \, \Omega_T,
	\label{eq:chns-mixed-d}
	\\
\partial_n \phi =\partial_n \mu = 0, \bu = {\bf 0} &\quad \text{on} \, \partial \Omega\times (0,T) ,
	\label{eq:chns-mixed-e}
	\end{align}
	\end{subequations}
where $M(\phi)>0$ is a mobility, $\eta = \frac{1}{Re}$ where $Re$ is the Reynolds number, $\gamma = \frac{1}{We^*}$ where $We^*$ is the modified Weber number that measures relative strengths of kenetic and surface energies, and $\mu$ is the chemical potential:
	\begin{align}
\mu := \delta_\phi E = \frac{1}{\varepsilon} \left(\phi^3 - \phi\right) - \varepsilon \Delta \phi.
	\end{align} 
Here $\delta_\phi E$ denotes the variational derivative of \eqref{eq:Ginzburg-Landau-energy-modified} with respect to $\phi$. The equilibria are the pure phases $\phi = \pm 1$. The boundary conditions are of local thermodynamic equilibrium and no-flux/no-flow/no-slip type. 

A weak formulation of \eqref{eq:chns-mixed-a} -- \eqref{eq:chns-mixed-e} may be written as follows: find $(\phi,\mu,\bu,p)$ such that
	\begin{subequations}
	\begin{align}
\phi \in& \, L^\infty\left(0,T;H^1(\Omega)\right)\cap  L^4\left(0,T;L^\infty(\Omega)\right),
	\\
\partial_t \phi \in& \, L^2 \left(0,T; H^{-1}(\Omega)\right), 
	\\
\mu \in& \, L^2 \left(0,T;H^1(\Omega)\right),
	\\
\bu \in& \, L^2 \left(0,T;{\bf H}^1_0(\Omega)\right) \cap L^\infty\left(0,T; {\bf L}^2(\Omega)\right),
	\\
\partial_t \bu \in& \, L^2 \left(0,T;{\bf H}^{-1}(\Omega)\right),
	\\
p \in& \, L^2\left(0,T;L^2_0(\Omega)\right),
	\end{align}
	\end{subequations}
and there hold for almost all $t\in (0,T)$ 
	\begin{subequations}
	\begin{align}
\langle \partial_t \phi ,\nu \rangle + \varepsilon \,\aiprd{\mu}{\nu} + \bform{\phi}{\bu}{\nu} &= 0, \quad \forall \,\nu \in H^1(\Omega),
	\label{eq:weak-chns-a} 
	\\
\iprd{\mu}{\psi}-\varepsilon \,\aiprd{\phi}{\psi} - \varepsilon^{-1}\iprd{\phi^3-\phi}{\psi}  &= 0, \quad \forall \,\psi\in H^1(\Omega),
	\label{eq:weak-chns-b}
	\\
\langle \partial_t \bu, \bv \rangle + \eta \, \aiprd{\bu}{\bv} + \Bform{\bu}{\bu}{\bv} - \cform{\bv}{p} - \gamma \, \bform{\phi}{\bv}{\mu} &=0, \quad \forall \, \bv \in {\bf H}^1_0(\Omega),
	\label{eq:weak-chns-c}
	\\
\cform{\bu}{q} &=0, \quad \forall \, q \in L^2_0(\Omega),
	\label{eq:weak-chns-d}
	\end{align}
	\end{subequations}
where
	\begin{align}
\aiprd{u}{v} := \iprd{\nabla u}{\nabla v}, \quad \bform{\psi}{\bv}{\nu} := \iprd{\nabla \psi \cdot \bv}{\nu}, 
	\\
\cform{\bv}{q} := \iprd{\nabla \cdot \bv}{q},\quad \Bform{\bu}{\bv}{{\bf w}} := \hf \left[\iprd{\bu \cdot \nabla \bv}{{\bf w}} - \iprd{\bu \cdot \nabla {\bf w}}{\bv} \right],
	\end{align}
with the ``compatible" initial data
	\begin{align}
\phi(0) &= \phi_0 \in H^2_N(\Omega) := \left\{v\in H^2(\Omega) \,\middle| \,\partial_n v = 0 \,\mbox{on} \,\partial\Omega \right\},
	\nonumber
	\\
\bu(0) &= \bu_0 \in {\bf V} := \left\{\bv \in {\bf H}^1_0(\Omega) | \iprd{\nabla \cdot \bv}{q} = 0, \forall q \in L^2_0(\Omega)\right\},
	\end{align}
and we have taken $M(\phi) \equiv 1$ for simplicity. Observe that the homogeneous Neumann boundary conditions associated with the phase variables $\phi$ and $\mu$ are natural in this mixed weak formulation of the problem. We define the space $L^2_0$ as the subspace of functions of $L^2$ that have mean zero. Furthermore, we state the following definitions of which the first is non-standard: $H^{-1}(\Omega) := \left(H^1(\Omega)\right)^*$, ${\bf H}_0^1(\Omega):= \left[H_0^1(\Omega)\right]^d$, ${\bf H}^{-1}(\Omega) := \left({\bf H}_0^1(\Omega)\right)^*$, and $\langle  \, \cdot \, , \, \cdot \, \rangle$ as the duality paring between $H^{-1}$ and $H^1$ in the first instance  and the duality paring between ${\bf H}^{-1}(\Omega)$ and $\left({\bf H}_0^1(\Omega)\right)^*$ in the second. The notation $\Phi(t) := \Phi(\, \cdot \, , t)\in X$ views a spatiotemporal function as a map from the time interval $[0,T]$ into an appropriate Banach space, $X$.  We use the standard notation for function space norms and inner products. In particular, we let $\norm{u}{}:=\norm{u}{L^2}$ and $\iprd{u}{v} := \iprd{u}{v}_{L^2}$, for all $u,v\in L^2(\Omega)$.

The existence of weak solutions to \eqref{eq:weak-chns-a} -- \eqref{eq:weak-chns-d} is well known.  See, for example,~\cite{liu03}. It is likewise straightforward to show that weak solutions of \eqref{eq:weak-chns-a} -- \eqref{eq:weak-chns-d} dissipate the energy \eqref{eq:Ginzburg-Landau-energy-modified}. In other words, \eqref{eq:chns-mixed-a} -- \eqref{eq:chns-mixed-e} is a mass-conservative gradient flow with respect to the energy \eqref{eq:Ginzburg-Landau-energy-modified}. Precisely, for any $t\in[0,T]$, we have the energy law
	\begin{equation}
E(\phi(t), \bu(t)) +\int_0^t \varepsilon\norm{\nabla\mu(s)}{}^2 + \frac{\eta}{\gamma} \norm{\nabla \bu(s)}{}^2 ds = E(\phi_0, \bu_0) ,
	\label{eq:pde-energy-law}
	\end{equation}
and where mass conservation (for almost every $t\in[0,T]$, $\iprd{\phi(t)-\phi_0}{1} = 0$) of the system \eqref{eq:weak-chns-a} -- \eqref{eq:weak-chns-d} is shown by observing that $\bform{\phi}{\bu}{1} = 0$, for all $\phi\in L^2(\Omega)$ and all $\bu \in {\bf V}$.  

Numerical methods for modeling two-phase flow via phase field approximation has been extensively investigated. See, for example, \cite{chen16a, chen16, elliott11, feng06, feng07b, feng12, grun13, grun14, kay08, kay07, kim04, liu03, liu16, shen10b, shen10a, shen15, wangx12, wangx13, wise10}, and the references therein).   Of the most recent, Shen and Yang~\cite{shen15} proposed two new numerical schemes for the Cahn-Hilliard-Navier-Stokes equations, one based on stabilization and the other based on convex splitting. Their new schemes have the advantage of being totally decoupled, linear, and unconditionally energy stable. Additionally, their schemes are adaptive in time and they provide numerical experiments which suggest that their schemes are at least first order accurate in time. However, no rigorous error analysis was presented. 

Abels~\emph{et.~al.}~\cite{abels12} introduce a thermodynamically consistent generalization to the Cahn-Hilliard-Navier-Stokes model for the case of non-matched densities based on a solenoidal velocity field.  The authors demonstrate that their model satisfies a free energy inequality and conserves mass. The work of Abels~\emph{et.~al.} builds on the pioneering paper of Lowengrub and Truskinovsky~\cite{lowengrub98} who use a mass-concentration formulation of the problem. Perhaps the fundamental difference between the approaches is that the model of Lowengrub and Truskinovsy end up with a velocity field that is not divergence free, in contrast with that of Abels~\emph{et.~al.}. For this reason, and others, developing suitable numerical schemes for the model in~\cite{lowengrub98} is a difficult task, but see the recent work of~\cite{guo14}. Garcke \emph{et.~al.}~\cite{garcke16} present a new time discretization scheme for the numerical simulation for the model in~\cite{abels12}. They show that their scheme satisfies a discrete in time energy law and go on to develop a fully discrete model which preserves that energy law. They are furthermore able to show existence of solutions to both the time discrete and fully discrete schemes. Again, however, no rigorous error analysis is undertaken for either of these schemes. Gr\"{u}n~\emph{et al.}~\cite{grun13,grun14} provide another numerical scheme for the non-matched density model and they carry out an abstract convergence analysis for their scheme. Rigorous error analysis (with, say, optimal order error bounds) for models with non-matched densities seems to be a difficult prospect, but a very interesting line of inquiry for the future.

Most of the papers referenced above present first order accurate in time numerical schemes. Second order in time numerical schemes provide an obvious advantage over first order schemes by decreasing the amount of numerical error. On the other hand, second-order (in time) methods are almost universally more difficult to analyze than first-order methods. A few such methods have been developed in recent years~\cite{chen16, han15, hintermueller13, hua11}. Most notably, Han and Wang~\cite{han15} present a second order in time, uniquely solvable, unconditionally stable numerical scheme for the \emph{CHNS} equations with match density. Their scheme is based on a second order convex splitting methodology for the Cahn-Hilliard equation and pressure-projection for the Navier-Stokes equation. The authors show that the scheme satisfies a modified discrete energy law which mimics the continuous energy law and prove that their scheme is uniquely solvable. However, no rigorous error analysis is presented and stability estimates are restricted to those gleaned from the energy law. The overall scheme is based on the Crank Nicholson time discretization and a second order Adams Bashforth extrapolation. Chen and Shen~\cite{chen16} have very recently refined the scheme of Han and Wang~\cite{han15}.

In this paper, we study a second-order in time mixed finite element scheme for the \emph{CHNS} system of equations with matched densities. The method essentially combines the recently analyzed second-order method for the Cahn-Hilliard equation from~\cite{diegel15b,guo16} and the pioneering second-order (in time) linear, Crank-Nicholson methodology for the Navier-Stokes equations found in~\cite{baker76}. The Cahn-Hilliard scheme from~\cite{diegel15b,guo16} is based on convex splitting and some key modifications of the Crank-Nicholson framework. The mixed finite element version of the scheme was analyzed rigorously in~\cite{diegel15b}.  The scheme herein is coupled, meaning the Cahn-Hilliard and Navier-Stokes must be solved simultaneously. But, the method is almost linear, with only a single weak nonlinearity present from the chemical potential equation. In particular, second order Adams-Bashforth extrapolations are used to linearize some terms and maintain the accuracy of the method, without compromising the unconditional energy stability and unconditional solvability of the scheme. The convergence analysis of a fully decoupled scheme, such as those in~\cite{chen16,han15} is far more challenging. The present work may be viewed as a first step towards analyzing such methods.

Theoretical justification for the convergence analysis and error estimates of numerical schemes applied to phase field models for fluid flow equations has attracted a great deal of attention in recent years. In particular, the recent work~\cite{diegel15} provides an analysis for an optimal error estimate (in the energy norms) for a first-order-accurate convex splitting finite element scheme applied to the Cahn-Hilliard-Nonsteady-Stokes system. The key point of that convergence analysis is the derivation of the maximum norm bound of the phase variable, which becomes available due to the discrete $\ell^2 (0,T; H^1)$ stability bound of the velocity field, at the numerical level. However, a careful examination shows that the same techniques from~\cite{diegel15} cannot be directly applied to the second-order-accurate numerical scheme studied in this paper. The primary difficulty is associated with the $3/4$ and $1/4$ coefficient distribution in the surface diffusion for the phase variable, at time steps $t^{n+1}$, $t^{n-1}$, respectively. In turn, an $\ell^\infty (0,T; H^2)$ estimate for the phase variable could not be derived via the discrete Gronwall inequality in the standard form. 

We therefore present an alternate approach to recover this $\ell^\infty (0,T; H^2)$ estimate for the phase field variable. A backward in time induction estimate for the $H^2$ norm of the phase field variable is applied. In addition, its combination with the $\ell^\infty (0,T; L^2)$ estimate for the chemical potential leads to an inequality involving a double sum term, with the second sum in the form of $\sum_{j=1}^m (\frac13)^{m-j}$. Subsequently, we apply a very non-standard discrete Gronwall inequality, namely Lemma~\ref{lem-discrete-gronwall-2} in Appendix~\ref{appen:A}, so that an $\ell^\infty (0,T; H^2)$ bound for the numerical solution of the phase variable is obtained. Moreover, the growth of this bound is at most linear in time, which is a remarkable result. 

It turns out that this stability bound greatly facilitates the second order convergence analysis in the energy norms for the numerical scheme presented in this paper. We point out that because of the $\ell^\infty (0,T; H^2)$ bound for the discrete phase variable, we are able to carry out the analysis on the Navier-Stokes part of the system that is much in the spirit of that which appears in Baker's groundbreaking paper~\cite{baker76}.  Due to the increased complexity of numerical calculations and the appearance of the nonlinear convection terms, a few more technical lemmas are required for the analysis included in this paper compared to the work presented in~\cite{diegel15} for the Cahn-Hilliard-Nonsteady-Stokes system. The use of these lemmas results in a numerical scheme which attains optimal convergence estimates in the appropriate energy norms: $\ell^\infty (0,T; H^1)$ for the phase variable and $\ell^2 (0,T; H^1)$ for the chemical potential. Moreover, such convergence estimates are unconditional: no scaling law is required between the time step size $\tau$ and the spatial grid size $h$. 

The remainder of the paper is organized as follows. In Section~\ref{sec:defn-and-properties}, we define our second order (in time) mixed finite element scheme and prove that the scheme is unconditionally stable and solvable with respect to both the time and space step sizes. In Section~\ref{sec:error estimates}, we provide a rigorous error analysis for the scheme under suitable regularity assumptions for the PDE solution. Finally, a few discrete Gronwall inequalities are reviewed and analyzed in Appendix~\ref{appen:A}.

	\section{A Second-Order-in-Time, Mixed Finite Element Scheme}
	\label{sec:defn-and-properties}
	
	\subsection{Definition of the Scheme}
	\label{subsec-defn}

Let $M$ be a positive integer and $0=t_0 < t_1 < \cdots < t_M = T$ be a uniform partition of $[0,T]$, with $\tau = t_{i+1}-t_{i}$ and $i=0,\ldots ,M-1$.  Suppose ${\mathcal T}_h = \left\{ K \right\}$ is a conforming, shape-regular, quasi-uniform family of triangulations of $\Omega$.  For $r\in\mathbb{Z}^+$, define $\mathcal{M}_r^h := \left\{v\in C^0(\Omega) \, \middle| \,v|_K \in {\mathcal P}_r(K), \,\forall \,\,  K\in \mathcal{T}_h \right\}\subset H^1(\Omega)$ and $\mathcal{M}_{r,0}^h:= \mathcal{M}_r^h \cap H_0^1(\Omega)$.

For a given positive integer $q$, we define the following:
	\begin{align*}
S_h &:= \mathcal{M}_q^h, 
	\\
\Soh &:= S_h\cap L_0^2(\Omega), 
	\\
\Xh &:= \left\{\bv \in \left[C^0(\Omega)\right]^d \ \middle| \ v_i \in \mathcal{M}_{q+1,0}^h, i = 1, \cdots, d \right\},
	\\
\Vh &:= \left\{ \bv \in \Xh \  \middle| \ \iprd{\nabla \cdot \bv}{w} = 0, \forall w \in \Soh\right\}.
	\end{align*}	
With the finite element spaces defined above, our mixed second-order convex splitting scheme is defined as follows:  for any $1\le m\le M$, given  $\phih^{m}, \phih^{m-1} \in S_h, \buh^m, \buh^{m-1} \in \Xh, p_h^{m} \in \Soh$ find $\phih^{m+1},\muh^{m+\hf} \in S_h, \buh^{m+1} \in \Xh,$ and $p_h^{m+1} \in \Soh$ such that
	\begin{subequations}
	\begin{align}
\iprd{\dtau \phih^{m+\hf}}{\nu} + \varepsilon \,\aiprd{\muh^{m+\hf}}{\nu} + \bform{\phitil^{m+\hf}}{\bbuh^{m+\hf}}{\nu} &= \, 0 , \forall \, \nu \in S_h,
	\label{eq:scheme-a}
	\\
\frac{1}{\varepsilon} \,\iprd{\chi\left(\phih^{m+1},\phih^m\right)}{\psi} - \frac{1}{\varepsilon} \iprd{\phitil^{m+\hf}}{\psi} + \varepsilon \,\aiprd{\phich^{m+\hf}}{\psi} - \iprd{\muh^{m+\hf}}{\psi} &= \, 0 , \forall \, \psi \in S_h,
	\label{eq:scheme-b}
	\\
\iprd{\dtau \buh^{m+\hf}}{\bv} + \eta \, \aiprd{\bbuh^{m+\hf}}{\bv} + \Bform{\butil^{m+\hf}}{\bbuh^{m+\hf}}{\bv} - \cform{\bv}{\bar{p}_h^{m+\hf}} &
	\nonumber
	\\
- \gamma \, \bform{\phitil^{m+\hf}}{\bv}{\muh^{m+\hf}} &= \, 0, \forall \,  \bv \in \Xh,
	\label{eq:scheme-c}
	\\
\cform{\bbuh^{m+\hf}}{q} &= \, 0, \forall \, q \in \Soh,
	\label{eq:scheme-d}
	\end{align}
	\end{subequations}
where
	\begin{align}
\dtau \phih^{m+\hf} &:= \frac{\phih^{m+1} - \phih^{m}}{\tau}, \quad \bphih^{m+\hf} := \hf \phih^{m+1} + \hf \phih^m, \quad \phitil^{m+\hf} := \frac32 \phih^m - \hf \phih^{m-1},
	\nonumber
	\\
\phich^{m+\hf} &:= \frac34 \phih^{m+1} + \frac14 \phih^{m-1}, \quad  \chi\left(\phih^{m+1},\phih^m\right) := \hf \left(\left(\phih^{m+1}\right)^2 + \left(\phih^{m}\right)^2\right) \bphih^{m+\hf}.
	\end{align}
The notation involving the pressure and velocity approximations are similar. For initial conditions we take
	\begin{align}
\phih^0 := R_h \phi_0, \ \phih^1 := R_h \phi(\tau), \quad \buh^0 := {\bf P}_h \bu_0, \ \buh^1 := {\bf P}_h \bu(\tau), \quad p_h^0 :=   P_h p_0,\ p_h^1 :=   P_h p(\tau),
	\end{align}
where $R_h: H^1(\Omega) \to S_h$ is the Ritz projection,
	\begin{equation}
\aiprd{R_h\phi - \phi}{\xi} = 0, \quad \forall \, \xi\in S_h, \quad \iprd{R_h \phi-\phi}{1}=0,
	\label{eq:Ritz-projection}
	\end{equation}
and $( {\bf P}_h, P_h) : {\bf V} \times L^2_0 \to \Vh \times \Soh$ is the Stokes projection,
	\begin{align}
\eta \, \aiprd{{\bf P}_h \bu - \bu}{\bv} - \cform{\bv}{P_h p - p} &= 0, \quad \forall \, \bv \in \Xh,
	\nonumber
	\\
\cform{{\bf P}_h \bu - \bu}{q} &= 0, \quad \forall \, q \in \Soh.
	\label{eq:Stokes-projection}
	\end{align}
It will be useful for our stability analyses to define the chemical potential at the $\frac{1}{2}$ time step via
	\begin{equation}
\iprd{\muh^{\hf}}{\psi} :=  \frac{1}{\varepsilon} \,\iprd{\chi\left(\phih^1,\phih^0\right)}{\psi} - \frac{1}{\varepsilon} \iprd{\bphih^{\hf}}{\psi} + \varepsilon \,\aiprd{\bphih^{\hf}}{\psi}, \quad  \forall \, \psi \in S_h.
	\label{eq:scheme-b-init}
	\end{equation}
We also define the residual function $\rho_h^\hf\in S_h$ that solves
	\begin{equation}
\iprd{ \rho_h^{\hf}}{\nu} :=    \iprd{\dtau \phih^{\hf}}{\nu}  + \varepsilon \,\aiprd{\muh^{\hf}}{\nu} + \bform{\bphih^{\hf}}{\bbuh^{\hf}}{\nu}, \quad  \forall \, \nu \in S_h.
	\label{eq:scheme-a-init}
	\end{equation}
While we do not expect the residual $\rho_h^\hf$ to be identically zero for finite $h,\tau > 0$, it will be stable in the relevant norms with the assumption of sufficiently regular PDE solutions.

	\begin{rem}
We have assumed exact expressions for $\phih^1$ and $\buh^1$. This is done to manage the length of the manuscript.  We can employ a separate initialization scheme, but the analysis becomes far more tedious. See, for example, \cite{diegel15b}. We point out that, because of the properties of the Ritz projection,
	\begin{equation}
\iprd{\phi_h^0}{1} = \iprd{\phi_h^1}{1},
	\end{equation}
under the natural assumption that $\iprd{\phi(0)}{1} = \iprd{\phi(\tau)}{1}$. Furthermore, note that this implies, $\iprd{ \dtau \phih^{\hf}}{1} = 0$.
	\end{rem}

	\begin{prop}
Suppose that $\psi\in S_h$ and $\bv \in \Vh$ are arbitrary. Then
	\begin{equation}
\bform{\psi}{\bv}{1} = 0.
	\end{equation}
	\end{prop}
	\begin{proof}
Using integration-by-parts, we get
	\begin{align}
\bform{\psi}{\bv}{1}= \iprd{\nabla\psi}{\bv} = -\iprd{\psi}{\nabla\cdot\bv} 
=  & \ -\overline{\psi}\iprd{1}{\nabla\cdot\bv} -\iprd{\psi-\overline{\psi}}{\nabla\cdot\bv} 
	\nonumber
	\\
= & \ -\overline{\psi}\cform{1}{\bv} -\cform{\psi-\overline{\psi}}{\bv} = 0.
	\end{align}
Observe that $\cform{1}{\bv} =0$ by the divergence theorem, using $\bv\cdot{\bf n}=0$ on $\partial\Omega$, and $\cform{\psi-\overline{\psi}}{\bv}=0$ since $\psi-\overline{\psi}\in\Soh$ and $\bv\in \Vh$.
	\end{proof}
	
	\begin{rem}
The last result relies on the fact that $\psi-\overline{\psi}\in \Soh$. In other words, the phase field space should be a subspace of the pressure space, which is restrictive. If this does not hold, mass conservation is lost. It may, however, be possible to prove what we wish using a variation of the trilinear form $b$. For example, we may take the alternate form
	\[
b(\psi,\bv,q) := \iprd{\nabla\psi\cdot\bv}{\nu}+\iprd{\nabla\cdot\bv}{\psi \nu}.
	\]
This allows us to decouple the pressure space from the phase space. The analysis of this case will be considered in a future work.
	\end{rem}

	\begin{rem}
For the Stokes projection, if the family of meshes satisfy certain reasonable properties, we have 
	\begin{equation}
\norm{{\bf P}_h {\bf u}-{\bf u}}{} + h\norm{\nabla\left({\bf P}_h {\bf u}-{\bf u}\right)}{} + h\norm{P_h p - p}{} \le C h^{s+1}\left(\left|{\bf u}\right|_{H^{s+1}} + \left|p\right|_{H^s}\right) ,
	\label{stokes-estimate}
	\end{equation}
provided that $(\bu,p)\in {\bf H}_0^1(\Omega)\cap {\bf H}^{s+1}(\Omega)\times H^s(\Omega)$, for all $0\le s \le q+1$. In fact, for our analysis, we do not need the optimal case $s = q+1$. We only require that the sub-optimal case $s = q$ holds, in other words, we will only assume $(\bu,p)\in {\bf H}_0^1(\Omega)\cap {\bf H}^{q+1}(\Omega)\times H^q(\Omega)$. See Assumption~\ref{asmp:initial-stability-3}.
	\end{rem}

Following similar arguments to what are given in \cite{diegel15}, we get the following theorem, which we state without proof:
	\begin{thm}
For any $1\le m\le M-1$, the fully discrete scheme \eqref{eq:scheme-a} -- \eqref{eq:scheme-d} is uniquely solvable and mass conservative, \emph{i.e.}, $\iprd{\phi_h^m-\phi^0}{1}{} = 0$.
	\end{thm}

	\subsection{Unconditional Energy Stability}
	\label{subsec-energy-stability}

We now show that the solutions to our scheme enjoy stability properties that are similar to those of the PDE solutions, and moreover, these properties hold regardless of the sizes of $h$ and $\tau$.  The first property, the unconditional energy stability, is a direct result of the convex decomposition. 

Consider the modified energy 
	\[
F(\phi, \psi, \bu) := E(\phi, \bu) + \frac{1}{4\varepsilon} \norm{\phi - \psi}{}^2 + \frac{\varepsilon}{8} \norm{\nabla \phi - \nabla \psi}{}^2,
	\]
where $E(\phi, \bu)$ is defined as above.
	
	\begin{lem}
	\label{lem-energy-law}
Let $(\phih^{m+1}, \muh^{m+\hf}, \buh^{m+1}, p_h^{m+1}) \in S_h\times S_h \times \Xh \times \Soh$ be the unique solution of  \eqref{eq:scheme-a} -- \eqref{eq:scheme-d}, for $1\le m\le M-1$.  Then the following energy law holds for any $h,\,  \tau >0$:
	\begin{align}
&F\left(\phih^{\ell+1}, \phih^{\ell}, \buh^{\ell+1} \right) + \tau \sum_{m=1}^\ell \left(\varepsilon \norm{\nabla\muh^{m+\hf}}{}^2 + \frac{\eta}{\gamma} \norm{\nabla \bbuh^{m+\hf}}{}^2 \right)
	\nonumber
	\\
&\hspace{.3cm}+  \sum_{m=1}^\ell \Bigg[ \frac{1}{4\varepsilon} \norm{\phih^{m+1} - 2 \phih^m + \phih^{m-1}}{}^2  + \frac{\varepsilon}{8} \norm{\nabla \phih^{m+1} - 2 \nabla \phih^m + \nabla \phih^{m-1}}{}^2 \Biggr] = F\left(\phih^1, \phih^0, \buh^1 \right),
	\label{eq:ConvSplitEnLaw}
	\end{align}
for all $1 \leq \ell \leq M-1$.
	\end{lem}

	\begin{proof}
Setting $\nu= \muh^{m+\hf}$ in \eqref{eq:scheme-a}, $\psi = \dtau\phih^{m+\hf}$ in \eqref{eq:scheme-b}, $\bv = \frac{1}{\gamma} \bbuh^{m+\hf}$ in \eqref{eq:scheme-c}, and $q = \frac{1}{\gamma} \bar{p}_h^{m+\hf}$ in \eqref{eq:scheme-d} gives
	\begin{align}
\iprd{\dtau \phih^{m+\hf}}{\muh^{m+\hf}} + \varepsilon \norm{\nabla\muh^{m+\hf}}{}^2 +  \bform{\phitil^{m+\hf}}{\bbuh^{m+\hf}}{\muh^{m+\hf}} &=   0,
	\label{eq:tested-energy-1}
	\\
	\nonumber
\frac{1}{\varepsilon} \,\iprd{\chi\left(\phih^{m+1},\phih^m\right)}{\dtau \phih^{m+\hf}} - \frac{1}{\varepsilon} \iprd{\phitil^{m+\hf}}{\dtau \phih^{m+\hf}} \quad &
	\\
+ \, \varepsilon \,\aiprd{\phich^{m+\hf}}{\dtau \phih^{m+\hf}}  - \iprd{\muh^{m+\hf}}{\dtau \phih^{m+\hf}} &= 0,
	\label{eq:tested-energy-2-a}
	\\
\frac{1}{\gamma} \iprd{\dtau \buh^{m+\hf}}{\bbuh^{m+\hf}} + \frac{\eta}{\gamma}  \, \norm{\nabla \bbuh^{m+\hf}}{}^2 + \frac{1}{\gamma} \Bform{\butil^{m+\hf}}{\bbuh^{m+\hf}}{\bbuh^{m+\hf}} \quad &
	\nonumber
	\\
- \, \frac{1}{\gamma} \cform{\bbuh^{m+\hf}}{\bar{p}_h^{m+\hf}} - \bform{\phitil^{m+\hf}}{\bbuh^{m+\hf}}{\muh^{m+\hf}} &= 0,
	\label{eq:tested-energy-2-b}
	\\
\frac{1}{\gamma} \cform{\bbuh^{m+\hf}}{\bar{p}_h^{m+\hf}} &= 0.
	\label{eq:tested-energy-2-c}
	\end{align}
Combining \eqref{eq:tested-energy-1} -- \eqref{eq:tested-energy-2-c}, using the following identities 
	\begin{align}
& \iprd{\chi\left(\phih^{m+1},\phih^m\right)}{\dtau \phih^{m+\hf}} - \iprd{\phitil^{m+\hf}}{\dtau \phih^{m+\hf}} = \, \frac1{4\tau} \left( \norm{\left(\phih^{m+1}\right)^2-1}{}^2 - \norm{\left(\phih^{m}\right)^2-1}{}^2\right)
	\nonumber
	\\
&\hspace{8.75cm} + \frac1{4\tau} \left( \norm{\phih^{m+1} - \phih^m}{}^2 - \norm{\phih^{m} - \phih^{m-1}}{}^2\right)
	\nonumber
	\\
&\hspace{8.75cm}+ \frac1{4\tau} \norm{\phih^{m+1} - 2 \phih^m + \phih^{m-1}}{}^2,
	\label{eq:identity-nonlinear}
	\\
&\aiprd{\phich^{m+\hf}}{\dtau\phih^{m+\hf}} = \, \frac{1}{2 \tau} \left( \norm{\nabla \phih^{m+1}}{}^2 - \norm{\nabla \phih^{m}}{}^2 \right)  + \frac{1}{8 \tau} \norm{\nabla \phih^{m+1} - 2 \nabla \phih^m + \nabla \phih^{m-1}}{}^2 
	\nonumber
	\\
&\hspace{3.75cm}+ \frac{1}{8 \tau} \left(\norm{\nabla \phih^{m+1} - \nabla \phih^m}{}^2 - \norm{\nabla \phih^{m} - \nabla \phih^{m-1}}{}^2 \right) , 
	\label{eq:identity-check}
	\end{align}
and applying the operator $\tau\sum_{m=1}^\ell$ to the combined equations, we get  \eqref{eq:ConvSplitEnLaw}.
	\end{proof} 

	\begin{asmp}
	\label{asmp:initial-stability-1}
From this point, we assume the following reasonable stabilities independent of $h$ and $\tau$:
	\begin{align*}
E\left(\phih^0, \buh^0\right) \le C, 
		\\
F\left(\phih^1, \phih^0, \buh^1\right) = E(\phih^1, \buh^1) + \frac{1}{4\varepsilon} \norm{\phih^1 - \phih^0}{}^2 + \frac{\varepsilon}{8} \norm{\nabla \phih^1 - \nabla \phih^0}{}^2 \le C,
	\\
\tau  \norm{\nabla\muh^{\hf}}{}^2 + \tau \norm{\nabla \bbuh^{\hf}}{}^2  \le C,
	\end{align*}
where $C>0$ is independent of $h$.
	\end{asmp}
	\begin{rem}
In the sequel, we will not track the dependences of the estimates on the interface parameter $\varepsilon>0$ or the viscosity $\eta>0$, though these may be important.
	\end{rem}
	
The next result follows from energy stability and Assumption~\ref{asmp:initial-stability-1}. We omit the proof.

	\begin{lem}
	\label{lem-a-priori-stability-trivial}
Let $(\phih^{m+1}, \muh^{m+\hf}, \buh^{m+1}, p_h^{m+1}) \in S_h\times S_h \times \Xh \times \Soh$ be the unique solution of  \eqref{eq:scheme-a} -- \eqref{eq:scheme-d}, for $1\le m\le M-1$.  Then the following estimates hold for any $h,\, \tau>0$:
	\begin{align}
\max_{0\leq m\leq M} \left[ \norm{\nabla\phih^m}{}^2 + \norm{\left( \phih^m\right)^2-1}{}^2 + \norm{\buh^m}{}^2 \right] &\leq C, 
	\label{eq:Linf-phi-discrete}
	\\
\max_{0\leq m\leq M}\left[\norm{\phih^m}{L^4}^4 +\norm{\phih^m}{}^2  +\norm{\phih^m}{H^1}^2 \right] &\leq C,
	\label{eq:Linf-H1-phi-discrete}
	\\
\max_{1\leq m\leq M}\left[\norm{\phih^m-\phih^{m-1}}{}^2 + \norm{\nabla \phih^m - \nabla \phih^{m-1}}{}^2\right] &\leq C,
	\label{eq:Linf-phi-diff-discrete}
	\\
\tau \sum_{m=0}^{M-1} \left[ \norm{\nabla\muh^{m+\hf}}{}^2  + \norm{\nabla \bbuh^{m+\hf}}{}^2 \right] &\leq C , 
	\label{eq:sum-mu-discrete}
	\\
\sum_{m=1}^{M-1} \left[ \norm{\phih^{m+1} - 2 \phih^m + \phih^{m-1}}{}^2 + \norm{\nabla \phih^{m+1} - 2 \nabla \phih^m + \nabla \phih^{m-1}}{}^2 \right] &\leq C , 
	\label{eq:sum-phi-discrete}
	\end{align}
for some constant $C>0$ that is independent of $h$, $\tau$, and $T$.
	\end{lem}

We are able to prove the next set of \emph{a priori} stability estimates without any restrictions on $h$ and $\tau$.   We define the discrete Laplacian, $\Delta_h: S_h \to \Soh$, as follows:  for any $v_h\in S_h$, $\Delta_h v_h\in \Soh$ denotes the unique solution to the problem
	\begin{equation}
\iprd{\Delta_h v_h}{\xi} = -\aiprd{v_h}{\xi},  \quad\forall \, \,\xi\in S_h.
	\label{eq:discrete-laplacian}
	\end{equation}
In particular, setting $\xi = \Delta_h v_h$ in \eqref{eq:discrete-laplacian}, we obtain  
	\begin{displaymath}
\norm{\Delta_h v_h}{}^2 = -\aiprd{v_h}{ \Delta_hv_h} .
	\end{displaymath}
We also need the following discrete Gagliardo-Nirenberg inequalities. See, for example,~\cite{heywood82,liu16}.
	\begin{prop}
If $\Omega$ is a convex polygonal or polyhedral domain, and $\Tauh$ is a globally quasi-uniform triangulation of $\Omega$, we have 
	\begin{align}
\norm{\psi_h}{L^\infty} \leq & \ C\norm{\Delta_h \psi_h}{}^{\frac{d}{2(6-d)}} \,\norm{\psi_h}{L^6}^{\frac{3(4-d)}{2(6-d)}} + C\norm{\psi_h}{L^6}, \quad \forall \, \psi_h \in S_h, \qquad d=2,3 ,
	\label{eq:infinity-bound}
	\\
\norm{\nabla \psi_h}{L^4}  \leq & \ C\left( \norm{\nabla \psi_h}{}+\norm{\Delta_h \psi_h}{}\right)^{\frac{d}{4}} \norm{\nabla \psi_h}{}^{\frac{4-d}{4}}, \quad \forall \, \psi_h \in S_h,\quad d= 2,3,
	\label{eq:grad-v-L4-bound}
	\end{align}
for some constant $C>0$ that is independent of $h$.
	\end{prop}

	\begin{asmp}
	\label{asmp:initial-stability-2}
From this point, we will assume that $\Omega$ is a convex polygonal or polyhedral domain, and $\Tauh$ is a globally quasi-uniform triangulation of $\Omega$. Furthermore, we assume the following initial stabilities hold: 
	\begin{equation}
\tau  \norm{\dtau \phih^{\hf}}{H^{-1}}^2 + \tau\norm{\dtau \phih^{\hf}}{-1,h}^2 +\tau \norm{\muh^{\hf}}{}^2  \le  C,
	\end{equation}
where $C>0$ is independent of $h$.
	\end{asmp}
See \cite{aristotelous13,diegel15,liu16} for a definition of the norm $\norm{\, \cdot\, }{-1,h}$.
	
	\begin{lem}
	\label{lem-improved-a-priori-stabilities}
Let $(\phih^{m+1}, \muh^{m+\hf}, \buh^{m+1}, p_h^{m+1}) \in S_h\times S_h \times \Xh \times \Soh$ be the unique solution of  \eqref{eq:scheme-a} -- \eqref{eq:scheme-d}, $1\le m \le M-1$. Then the following estimates hold for any $h,\, \tau >0$:
	\begin{align}
\tau \sum_{m=0}^{M-1} \Bigg[ \norm{\dtau \phih^{m+\hf}}{H^{-1}}^2 +\norm{\dtau \phih^{m+\hf}}{-1,h}^2 \Bigg] &\le C,
	\label{eq:sum-3D-good-dtau}
	\\
\tau \sum_{m=0}^{M-1}  \norm{\muh^{m+\hf}}{}^2 & \le C (T+1),
	\label{eq:sum-3D-good-mu}
	\\
\varepsilon \norm{ \Delta_h \phich^{m+\hf}}{}^2 & \leq \norm{\muh^{m+\hf}}{}^2  + C, \quad \forall \ 1 \le m \le M-1,
	\label{eq:discrete-laplacian-phi-mu-equiv}
	\\
\varepsilon \norm{ \Delta_h \bphih^{\hf}}{}^2 & \leq \norm{\muh^{\hf}}{}^2  + C, 
	\label{eq:discrete-laplacian-phi-mu-equiv-init}
	\\
\tau \sum_{m=1}^{M-1} \Bigg[	\norm{ \Delta_h \phich^{m+\hf}}{}^2  + \norm{\phich^{m+\hf}}{L^\infty}^{\frac{4(6-d)}{d}} \Bigg] &\le C (T+1),
	\label{eq:sum-3D-good-phicheck}
	\end{align}
for some constant $C>0$ that is independent of $h$, $\tau$, and $T$.
	\end{lem}
The proof of Lemma~\ref{lem-improved-a-priori-stabilities} is very similar to proofs of \cite[Lemma~2.7]{diegel15b} and \cite[Lemma~2.13]{diegel15}. We omit the details for the sake of brevity.

	\subsection{Unconditional $\ell^\infty(0,T;L^\infty)$ Stability of the Discrete Phase Variable}
	\label{subsec-max-norm-stability}

	\begin{lem}
	\label{lem:induction}
Let $(\phih^{m+1}, \muh^{m+\hf}, \buh^{m+1}, p_h^{m+1}) \in S_h\times S_h \times \Xh \times \Soh$ be the unique solution of  \eqref{eq:scheme-a} -- \eqref{eq:scheme-d}, for $1\le m\le M-1$. Then the following estimates hold for any $h,\, \tau >0$:
	\begin{align}
\norm{\Delta_h \phih^{2m}}{}^2 & \le \frac{8}{3} \sum_{k=1}^m \left(\frac{1}{3}\right)^{k-1} \norm{\Delta_h \phich^{2k-\hf}}{}^2 + \left(\frac13\right)^m \cdot \norm{\Delta_h \phih^0}{}^2,
	\label{eq:induction-even}
	\\
\norm{\Delta_h \phih^{2m+1}}{}^2 & \le \frac{8}{3} \sum_{k=1}^m \left(\frac{1}{3}\right)^{k-1} \norm{\Delta_h \phich^{(2k+1)-\hf}}{}^2 + \left(\frac13\right)^m \cdot \norm{\Delta_h \phih^1}{}^2.
	\label{eq:induction-odd}
	\end{align}
	\end{lem}

	\begin{proof}
Using the definition of $\phich^{m+\hf}$, for $1\le m\le M-1$, we have the following inequality:
	\begin{eqnarray}  
\norm{\Delta_h \phich^{m+\hf}}{}^2 &=& \norm{\Delta_h \left(\frac34 \phih^{m+1} + \frac14 \phih^{m-1}\right)}{}^2 \nonumber 
	\\
&=& \frac{9}{16} \norm{\Delta_h \phih^{m+1}}{}^2 + \frac38 \iprd{\Delta_h \phih^{m+1}}{\Delta_h \phih^{m-1}} + \frac1{16} \norm{\Delta_h \phih^{m-1}}{}^2  \nonumber 
	\\
&\ge& \frac{9}{16} \norm{\Delta_h \phih^{m+1}}{}^2 - \frac{3}{16} \norm{\Delta_h \phih^{m+1}}{}^2 - \frac{3}{16} \norm{\Delta_h \phih^{m-1}}{}^2 + \frac{1}{16} \norm{\Delta_h \phih^{m-1}}{}^2 \nonumber 
	\\
&=& \frac{3}{8} \norm{\Delta_h \phih^{m+1}}{}^2 - \frac{1}{8} \norm{\Delta_h \phih^{m-1}}{}^2. 
\label{estimate:induction-odd}
	\end{eqnarray}
Its repeated use gives the result.
	\end{proof}	
	
	\begin{asmp}
	\label{asmp:initial-stability-2-b}
From this point on, we assume the following initial stabilities 
	\begin{equation}
\tau  \norm{\dtau \phih^{\hf}}{}^2 + \norm{\muh^{\hf}}{}^2+ \norm{\rho_h^{\hf}}{}^2+\norm{\Delta_h \phih^0}{}^2+\norm{\Delta_h \phih^1}{}^2  \le  C,
	\end{equation}
where $C>0$ is independent of $h$.
	\end{asmp}

We are now ready to show the main result for this section.

\begin{lem}
	\label{lem:a-priori-stability-dtau-mu}
Let $(\phih^{m+1}, \muh^{m+\hf}, \buh^{m+1}, p_h^{m+1}) \in S_h\times S_h \times \Xh \times \Soh$ be the unique solution of  \eqref{eq:scheme-a} -- \eqref{eq:scheme-d}, for $1\le m\le M-1$. Then the following estimates hold for any $h,\, \tau >0$:
	\begin{align}
\tau \sum_{m=0}^{M-1} \norm{\dtau\phih^{m+\hf}}{}^2 & \le C (T+1),
	\label{eq:sum-dtau-phi}
	\\
\max_{0\le m\le M-1} \norm{\muh^{m+\hf}}{}^2 + \max_{1\le m\le M-1} \bigg[ \norm{\Delta_h \phich^{m+\hf}}{}^2  + \norm{\phich^{m+\hf}}{L^\infty}^{\frac{4(6-d)}{d}} \bigg] &\le C(T+1),
	\label{eq:Linf-mu-phich} 
	\\
\max_{0\le m\le M} \bigg[ \norm{\Delta_h \phih^m}{}^2  + \norm{\phih^m}{L^\infty}^{\frac{4(6-d)}{d}} \bigg] &\le C(T+1),
	\label{eq:Linf-phi} 
	\end{align}
for some constant $C>0$ that is independent of $h$, $\tau$, and $T$.
	\end{lem}
	
	\begin{proof}
The proof will be completed in two parts. 
	
\noindent \textbf{Part 1:} ($m=1$) Subtracting \eqref{eq:scheme-b-init} from \eqref{eq:scheme-b} with $m=1$, we obtain
	\begin{align}
\iprd{\muh^{\frac32} - \muh^{\hf}}{\psi} =& \, \varepsilon \,\aiprd{\phich^{\frac32} - \bphih^{\hf}}{\psi} - \frac{1}{\varepsilon} \iprd{ \phitil^{\frac32}-\bphih^{\hf}}{\psi} + \frac{1}{\varepsilon} \iprd{\chi\left(\phih^2, \phih^1\right) - \chi\left(\phih^1, \phih^0\right)}{\psi} 
	\nonumber
	\\
=&\, \varepsilon \, \aiprd{\frac34 \tau \dtau \phih^{\frac32} + \frac14 \tau \dtau \phih^{\hf}}{\psi} - \frac{1}{\varepsilon} \iprd{\tau\dtau\phih^{\hf}}{\psi} 
	\nonumber
	\\
&+ \frac{1}{\varepsilon} \iprd{\chi\left(\phih^2, \phih^1\right) - \chi\left(\phih^1, \phih^0\right)}{\psi} .
	\label{eq:scheme-b-staggered-first-step}	
	\end{align}
Additionally, we take a weighted average of \eqref{eq:scheme-a} with $m=1$ and \eqref{eq:scheme-a-init} with the weights $\frac34$ and $\frac14$, respectively, to obtain,
	\begin{align}
\iprd{\frac34 \dtau \phih^{\frac32} + \frac14 \dtau \phih^{\hf}}{\nu} = & \ - \varepsilon \, \aiprd{\frac34 \muh^{\frac32} + \frac14 \muh^{\hf}}{\nu} - \frac34 \bform{\phitil^{\frac32}}{\bbuh^{\frac32}}{\nu} 
	\nonumber
	\\
& - \frac14 \bform{\bphih^\hf}{\bbuh^{\hf}}{\nu} + \frac{1}{4}\iprd{\rho_h^\hf}{\nu}.
	\label{eq:scheme-a-staggered-first-step}	
	\end{align}
Taking $\psi = \frac34 \muh^{\frac32} + \frac14 \muh^{\hf}$ in \eqref{eq:scheme-b-staggered-first-step}, $\nu = \frac{3\tau}{4} \dtau \phih^{\frac32} + \frac{\tau}{4} \dtau \phih^{\hf}$ in \eqref{eq:scheme-a-staggered-first-step}, and adding the results yields 
	\begin{align*}
& \hspace{-0.55in} \iprd{\muh^{\frac32} - \muh^{\hf}}{\frac34 \muh^{\frac32} + \frac14 \muh^{\hf}} +  \tau \norm{\frac34 \dtau \phih^{\frac32} + \frac14 \dtau \phih^{\hf}}{}^2 
	\\
= & \, - \frac{1}{\varepsilon} \iprd{\phih^1 - \phih^0}{\frac34 \muh^{\frac32} + \frac14 \muh^{\hf}} + \frac{1}{4 \varepsilon} \iprd{\chi\left(\phih^2, \phih^1\right) - \chi\left(\phih^1, \phih^0\right)}{3 \muh^{\frac32} + \muh^{\hf}} 
	\\
&- \frac34 \bform{\phitil^{\frac32}}{\bbuh^{\frac32}}{ \frac{3\tau}{4} \dtau \phih^{\frac32} + \frac{\tau}{4} \dtau \phih^{\hf}} - \frac14 \bform{\bphih^\hf}{\bbuh^{\hf}}{ \frac{3\tau}{4} \dtau \phih^{\frac32} + \frac{\tau}{4} \dtau \phih^{\hf}}
	\\
& +  \frac{1}{4}\iprd{\rho_h^\hf}{ \frac{3\tau}{4} \dtau \phih^{\frac32} + \frac{\tau}{4} \dtau \phih^{\hf}}
	\\
\le& \ \frac14 \norm{\muh^{\frac32}}{}^2 + C \norm{\muh^{\hf}}{}^2 + C \norm{\phih^1}{}^2 + C \norm{\phih^0}{}^2 + C \norm{\chi\left(\phih^2, \phih^1\right)}{}^2 + C \norm{\chi\left(\phih^1, \phih^0\right)}{}^2 
	\\
& + \frac{3\tau}{4} \norm{\nabla\phitil^{\frac32}}{L^4} \norm{\bbuh^{\frac32}}{L^4} \norm{\frac{3}{4} \dtau \phih^{\frac32}+ \frac{1}{4} \dtau \phih^{\hf}}{} + \frac{\tau}{4}  \norm{\nabla\bphih^\hf}{L^4} \norm{\bbuh^{\hf}}{L^4} \norm{\frac{3}{4} \dtau \phih^{\frac32} + \frac{1}{4} \dtau \phih^{\hf}}{} 
	\\
& + \frac{\tau}{4}  \norm{\rho_h^{\hf}}{} \norm{\frac{3}{4} \dtau \phih^{\frac32} + \frac{1}{4} \dtau \phih^{\hf}}{}
	\\
\le& \  C + \frac14 \norm{\muh^{\frac32}}{}^2 + C \tau \norm{\frac{3}{4} \dtau \phih^{\frac32}+ \frac{1}{4} \dtau \phih^{\hf}}{} \norm{\nabla \bbuh^{\frac32}}{} \left(\norm{\nabla\phitil^{\frac32}}{} + \norm{\Delta_h \phitil^{\frac32}}{}\right)
	\\
&+ C \tau \norm{\frac{3}{4} \dtau \phih^{\frac32}+ \frac{1}{4} \dtau \phih^{\hf}}{} \norm{\nabla \bbuh^{\frac12}}{} \left(\norm{\nabla\bphih^{\frac12}}{} + \norm{\Delta_h \bphih^{\frac12}}{}\right) + \frac{\tau}{6} \norm{\frac{3}{4} \dtau \phih^{\frac32} + \frac{1}{4} \dtau \phih^{\hf}}{}
	\\
\le& \ C+ \frac14 \norm{\muh^{\frac32}}{}^2 + \frac{3\tau}{6} \norm{\frac34 \dtau \phih^{\frac32}+ \frac14 \dtau \phih^{\hf}}{}^2 + C \tau \norm{\nabla \bbuh^{\frac32}}{}^2 + C \tau \norm{\nabla \bbuh^{\frac12}}{}^2
	\\
\le& \  C+ \frac14 \norm{\muh^{\frac32}}{}^2 + \frac{\tau}{2} \norm{\frac34 \dtau \phih^{\frac32}+ \frac14 \dtau \phih^{\hf}}{}^2 + C \tau \norm{\nabla \bbuh^{\frac32}}{}^2,
	\end{align*}
where we have used Young's inequality, the embedding $H^1 \hookrightarrow L^6$, estimates \eqref{eq:Linf-H1-phi-discrete}  and \eqref{eq:grad-v-L4-bound} and Assumption~\ref{asmp:initial-stability-2-b}. Considering Assumption \ref{asmp:initial-stability-2-b}, estimate~\eqref{eq:sum-mu-discrete}, and the following estimates
	\begin{align*}
\norm{\frac34 \dtau \phih^{\frac32} + \frac14 \dtau \phih^{\hf}}{}^2 \ge & \ \frac38 \norm{\dtau \phih^{\frac32}}{}^2 - \frac18 \norm{\dtau \phih^{\hf}}{}^2  \, \, \mbox{(similar to (\ref{estimate:induction-odd})) } ,   
	\\
\iprd{\muh^{\frac32} - \muh^{\hf}}{\frac34 \muh^{\frac32} + \frac14 \muh^{\hf}} = & \  \frac34 \norm{\muh^{\frac32}}{}^2 - \hf \iprd{\muh^{\frac32}}{\muh^{\hf}} - \frac14 \norm{\muh^{\hf}}{}^2
\ge \hf \norm{\muh^{\frac32}}{}^2 - \hf \norm{\muh^{\hf}}{}^2,
	\end{align*}
we have, 
	\begin{align}
\frac14 \norm{\muh^{\frac32}}{}^2 + \frac{3\tau}{16} \norm{\dtau \phih^{\frac32}}{}^2 \le C \norm{\muh^{\hf}}{}^2 + \frac{\tau}{16} \norm{\dtau \phih^{\hf}}{}^2 + C \tau \norm{\nabla \bbuh^{\frac32}}{}^2  + C \le C.
	\label{eq:mu-control-second-step}
	\end{align}

Now, using \eqref{eq:discrete-laplacian-phi-mu-equiv}, \eqref{eq:infinity-bound}, the embedding $H^1(\Omega) \hookrightarrow L^6(\Omega)$, and \eqref{eq:Linf-H1-phi-discrete}, we have
	\begin{align*}
\norm{\Delta_h \phich^{\frac32}}{}^2 + \norm{\phich^\frac32}{L^\infty}^{\frac{4(6-d)}{d}} \le C.
	\end{align*}
Using Lemma \ref{lem:induction}, \eqref{eq:infinity-bound}, the embedding $H^1(\Omega) \hookrightarrow L^6(\Omega)$, and \eqref{eq:Linf-H1-phi-discrete}, we obtain
	\begin{align*}
\norm{\Delta_h \phih^{2}}{}^2 + \norm{\phih^2}{L^\infty}^{\frac{4(6-d)}{d}} \le C.
	\end{align*}

\noindent \textbf{Part 2:} ($2 \le m \le M-1$)\\

For $2 \le m \le M-1$, we subtract \eqref{eq:scheme-b} from itself at consecutive time steps to obtain
	\begin{align}
\iprd{\muh^{m+\hf} - \muh^{m-\hf}}{\psi} =& \, \varepsilon \,\aiprd{\phich^{m+\hf} - \phich^{m-\hf}}{\psi} - \frac{1}{\varepsilon} \iprd{\phitil^{m+\hf} - \phitil^{m-\hf}}{\psi}
	\nonumber
	\\
&+ \frac{1}{\varepsilon} \iprd{\chi\left(\phih^{m+1}, \phih^m\right) - \chi\left(\phih^m, \phih^{m-1}\right)}{\psi} 
	\nonumber
	\\
=&\, \varepsilon \,\aiprd{\frac34 \tau \dtau \phih^{m+\hf} + \frac14 \tau \dtau \phih^{m-\frac32}}{\psi} - \frac{1}{\varepsilon} \iprd{\frac32 \tau \dtau \phih^{m-\hf} - \hf \tau \dtau \phih^{m-\frac32}}{\psi}
	\nonumber
	\\
&+ \frac{1}{4 \varepsilon} \iprd{\omega^m_h \left(\phih^{m+1} - \phih^{m-1}\right)}{\psi} , 
	\label{eq:scheme-b-staggered}
	\end{align}
for all $\psi \in S_h$, where 
	\begin{align*}
\omega^m_h := \omega\left(\phih^{m+1}, \phih^m, \phih^{m-1}\right) := \left(\phih^{m+1}\right)^2 + \left(\phih^{m}\right)^2 + \left(\phih^{m-1}\right)^2 + \phih^{m+1} \phih^m + \phih^{m+1} \phih^{m-1} + \phih^m \phih^{m-1}.
	\end{align*}
We note that using the $H^1(\Omega) \hookrightarrow L^6(\Omega)$ embedding, we achieve the following bound,  
	\begin{align*}
\norm{\omega^m_h}{L^3} =& \, \norm{\left(\phih^{m+1}\right)^2 + \left(\phih^{m}\right)^2 + \left(\phih^{m-1}\right)^2 + \phih^{m+1} \phih^m + \phih^{m+1} \phih^{m-1} + \phih^m \phih^{m-1}}{L^3}
	\\
\le& \, C \norm{\phih^{m+1}}{L^6}^2  + C \norm{\phih^{m}}{L^6}^2 + C \norm{\phih^{m-1}}{L^6}^2
	\\
\le& \, C \norm{\phih^{m+1}}{H^1}^2 + C \norm{\phih^{m}}{H^1}^2 + C \norm{\phih^{m-1}}{H^1}^2 \le C . 
	\end{align*}

Now, for all $2 \le m \le M-1$, we take a weighted average of the $m+\hf$ and $m-\frac32$ time steps with the weights $\frac34$ and $\frac14$ of \eqref{eq:scheme-a}, respectively, to obtain,
	\begin{align}
	\nonumber
\iprd{\frac34 \dtau \phih^{m+\hf} + \frac14 \dtau \phih^{m-\frac32}}{\nu} =& \, - \varepsilon \, \aiprd{\frac34 \muh^{m+\hf} + \frac14 \muh^{m-\frac32}}{\nu} - \frac34 \bform{\phitil^{m+\hf}}{\bbuh^{m+\hf}}{\nu} 
	\\
&- \frac14 \bform{\phitil^{m-\frac32}}{\bbuh^{m-\frac32}}{\nu} ,  \, \, \forall  \nu \in S_h . 
	\label{eq:scheme-a-staggered}
	\end{align}

Taking $\psi = \frac34 \muh^{m+\hf} + \frac14 \muh^{m-\frac32}$ in \eqref{eq:scheme-b-staggered}, $\nu = \tau \left(\frac34 \dtau \phih^{m+\hf} + \frac14 \dtau \phih^{m-\frac32} \right)$ in \eqref{eq:scheme-a-staggered}, and adding the results yields 
	\begin{align*}
& \hspace{-0.3in} \iprd{\muh^{m+\hf} - \muh^{m-\hf}}{\frac34 \muh^{m+\hf} + \, \frac14 \muh^{m-\frac32}} + \tau \norm{\frac34 \dtau \phih^{m+\hf} + \frac14 \dtau \phih^{m-\frac32}}{}^2 
	\\
=& \, -\frac{\tau}{\varepsilon} \iprd{\frac32 \dtau \phih^{m-\hf} - \hf \dtau \phih^{m-\frac32}}{\frac34 \muh^{m+\hf} + \frac14 \muh^{m-\frac32}}
+ \frac{\tau}{4 \varepsilon} \iprd{\omega^m_h \dtau \phih^{m+\hf} }{\frac34 \muh^{m+\hf} + \frac14 \muh^{m-\frac32}}
	\\
&+ \frac{\tau}{4 \varepsilon} \iprd{\omega^m_h \dtau \phih^{m-\hf}}{\frac34 \muh^{m+\hf} + \frac14 \muh^{m-\frac32}}
- \frac{3\tau}{4} \bform{\phitil^{m+\hf}}{\bbuh^{m+\hf}}{\frac34 \dtau \phih^{m+\hf} + \frac14 \dtau \phih^{m-\frac32}} 
	\\
&- \frac{\tau}{4} \bform{\phitil^{m-\frac32}}{\bbuh^{m-\frac32}}{\frac34 \dtau \phih^{m+\hf} + \frac14 \dtau \phih^{m-\frac32}}
	\\
\le& \, \frac{3 \tau}{8 \varepsilon} \norm{\dtau \phih^{m-\hf}}{} \cdot \norm{3\muh^{m+\hf} + \muh^{m-\frac32}}{}
+ \frac{\tau}{8 \varepsilon} \norm{\dtau \phih^{m-\frac32}}{} \cdot \norm{3 \muh^{m+\hf} + \muh^{m-\frac32}}{}
	\\
&+ \frac{\tau}{16 \varepsilon} \norm{\omega^m_h}{L^3} \left( \norm{\dtau \phih^{m+\hf}}{} \cdot \norm{3\muh^{m+\hf} + \muh^{m-\frac32}}{L^6}
+ \norm{\dtau \phih^{m-\hf}}{} \cdot \norm{3\muh^{m+\hf} + \muh^{m-\frac32}}{L^6} \right) 
	\\
&- \frac{3\tau}{4} \bform{\phitil^{m+\hf}}{\bbuh^{m+\hf}}{\frac34 \dtau \phih^{m+\hf} + \frac14 \dtau \phih^{m-\frac32}} 
- \frac{\tau}{4} \bform{\phitil^{m-\frac32}}{\bbuh^{m-\frac32}}{\frac34 \dtau \phih^{m+\hf} + \frac14 \dtau \phih^{m-\frac32}}.
	\end{align*}
Now we bound the trilinear form $b(\, \cdot \, , \, \cdot \, , \, \cdot \, )$. Using \eqref{eq:grad-v-L4-bound}, H\"older's inequality, and (\ref{eq:Linf-phi-discrete}), the following estimates are available: 
	\begin{align}
& \hspace{-0.3in} \left| \bform{\phitil^{m+\hf}}{\bbuh^{m+\hf}}{\frac34 \dtau \phih^{m+\hf} + \frac14 \dtau \phih^{m-\frac32}} \right| \le  \norm{\nabla\phitil^{m+\hf}}{L^4}\norm{\bbuh^{m+\hf}}{L^4}\norm{\frac34 \dtau \phih^{m+\hf} + \frac14 \dtau \phih^{m-\frac32}}{}
	\nonumber
	\\
\le& \,  C \norm{\frac34 \dtau \phih^{m+\hf} + \frac14 \dtau \phih^{m-\frac32}}{} \norm{\nabla \bbuh^{m+\hf}}{}
 \left(\norm{\nabla\phitil^{m+\hf}}{} + \norm{\Delta_h \phitil^{m+\hf}}{}\right)
	\nonumber
	\\
\le& \,  \frac{1}{2} \norm{\frac34 \dtau \phih^{m+\hf} + \frac14 \dtau \phih^{m-\frac32}}{}^2 + C \norm{\nabla \bbuh^{m+\hf}}{}^2 
+ C \norm{\nabla \bbuh^{m+\hf}}{}^2 \norm{\Delta_h \phitil^{m+\hf}}{}^2 , 
	\label{eq:chns-b-form-stability-estimate-b}
 	\end{align}
and, similarly, 
	\begin{align}
& \hspace{-0.5in} \left| \bform{\phitil^{m-\frac32}}{\buh^{m-\frac32}}{\frac34 \dtau \phih^{m+\hf} + \frac14 \dtau \phih^{m-\frac32}} \right| \nonumber 
\\
\le & \,  \frac{1}{2} \norm{\frac34 \dtau \phih^{m+\hf} + \frac14 \dtau \phih^{m-\frac32}}{}^2 + C \norm{\nabla \bbuh^{m-\frac32}}{}^2 
+ C \norm{\nabla \bbuh^{m-\frac32}}{}^2 \norm{\Delta_h \phitil^{m-\frac32}}{}^2 , 
	\label{eq:chns-b-form-stability-estimate-c}
	\end{align}
for any $2 \le m \le M-1$. 	 Therefore, we arrive at 
\begin{align}
&\hspace{-0.5in} \iprd{\muh^{m+\hf} - \muh^{m-\hf}}{\frac34 \muh^{m+\hf} + \frac14 \muh^{m-\frac32}} + \frac{\tau}{2} \norm{\frac34 \dtau \phih^{m+\hf} + \frac14 \dtau \phih^{m-\frac32}}{}^2 
	\nonumber
	\\
\le& \, \frac{\tau}{8} \norm{\dtau \phih^{m+\hf}}{}^2 + \frac{\tau}{32} \norm{\dtau \phih^{m-\hf}}{}^2 + \frac{\tau}{32} \norm{\dtau \phih^{m-\frac32}}{}^2 
+ C \tau \norm{\muh^{m+\hf}}{H^1}^2 + C \tau \norm{\muh^{m-\frac32}}{H^1}^2 \nonumber 
\\
& + C \tau \norm{\nabla \bbuh^{m+\hf}}{}^2 + C \tau \norm{\nabla \bbuh^{m+\hf}}{}^2 \norm{\Delta_h \phitil^{m-\hf}}{}^2 \nonumber 
\\
& + C \tau \norm{\nabla \bbuh^{m-\frac32}}{}^2 
+ C \tau \norm{\nabla \bbuh^{m-\frac32}}{}^2 \norm{\Delta_h \phitil^{m-\frac32}}{}^2 , \, \, 
\forall \,  2 \le m \le M-1 . 
	\label{eq:Linf-muh-step}
	\end{align}

Furthermore, we use Lemma \ref{lem:induction} and \eqref{eq:discrete-laplacian-phi-mu-equiv} to derive the following inequalities: 
	\begin{align}
\norm{\Delta_h \phitil^{(2k+1)+\hf}}{}^2 &\le C \norm{\Delta_h \phih^{2k+1}}{}^2 + C \norm{\Delta_h \phih^{2k}}{}^2 
	\nonumber
	\\
&\le C \sum_{j=1}^k \left( \frac13 \right)^{k-j} \left( \norm{\Delta_h \phich^{(2j+1)-\hf}}{}^2 + \norm{\Delta_h \phich^{(2j)-\hf}}{}^2 \right) + C 
	\nonumber
	\\
&\le C \sum_{j=1}^k \left( \frac13 \right)^{k-j} \left( \norm{\muh^{(2j+1)-\hf}}{}^2 + \norm{\muh^{(2j)-\hf}}{}^2 \right) + C , 	
	\label{eq:discrete-laplacian-phitil-odd}
	\\
\norm{\Delta_h \phitil^{(2k)+\hf}}{}^2 &\le C \norm{\Delta_h \phih^{2k}}{}^2 + C \norm{\Delta_h \phih^{2k-1}}{}^2 
	\nonumber
	\\
&\le C  \sum_{j=1}^k \left( \frac13 \right)^{k-j} \left( \norm{\Delta_h \phich^{(2j)-\hf}}{}^2 + \norm{\Delta_h \phich^{(2j-1)-\hf}}{}^2 \right) + C
	\nonumber
	\\
&\le C  \sum_{j=1}^k \left( \frac13 \right)^{k-j} \left( \norm{\muh^{(2j)-\hf}}{}^2 +  \norm{\muh^{(2j-1)-\hf}}{}^2 \right) + C.
	\label{eq:discrete-laplacian-phitil-even}
	\end{align}

Applying $\sum_{m=2}^{\ell}$ to \eqref{eq:Linf-muh-step} and using the following properties
	\begin{align*}
\iprd{\muh^{m+\hf} - \muh^{m-\hf}}{\frac34 \muh^{m+\hf} + \frac14 \muh^{m-\frac32}} =& \, \hf \iprd{\muh^{m+\hf} - \muh^{m-\hf}}{\muh^{m+\hf} + \muh^{m-\hf}}
	\\
&+ \frac14 \iprd{\muh^{m+\hf} - \muh^{m-\hf}}{\muh^{m+\hf} - 2\muh^{m-\hf} + \muh^{m-\frac32}}
	\\
=& \, \hf \norm{\muh^{m+\hf}}{}^2 - \hf \norm{\muh^{m-\hf}}{}^2 + \frac18 \norm{\muh^{m+\hf} - \muh^{m-\hf}}{}^2 
	\\
&- \frac18 \norm{\muh^{m-\hf} - \muh^{m-\frac32}}{}^2 + \frac18 \norm{\muh^{m+\hf} - 2 \muh^{m-\hf} + \muh^{m-\frac32}}{}^2,
	\\
\norm{\frac34 \dtau \phih^{m+\hf} + \frac14 \dtau \phih^{m-\frac32}}{}^2 \ge& \, \frac38 \norm{\dtau \phih^{m+\hf}}{}^2 - \frac18 \norm{\dtau \phih^{m-\frac32}}{}^2
\, \, \mbox{(similar to (\ref{estimate:induction-odd})) }  , 
	\end{align*}
we conclude that 
	\begin{align*}
\hf \norm{\muh^{\ell+\hf}}{}^2 + & \ \frac{\tau}{16} \sum_{m=2}^{\ell} \norm{\dtau \phih^{m+\hf}}{}^2 
	\nonumber
	\\
\le &  \ \frac18 \norm{\muh^{\frac32} - \muh^{\hf}}{}^2 + \frac{\tau}{32} \norm{\dtau \phih^{\frac32}}{}^2 + \frac{5 \tau}{32} \norm{\dtau \phih^{\hf}}{}^2  
	\\
& + C \tau \sum_{m=0}^{\ell}\norm{\muh^{m+\hf}}{H^1}^2 + C \tau \sum_{m=0}^{\ell}\norm{\nabla \bbuh^{m+\hf}}{}^2 \cdot \sum_{j=1}^m \left( \frac13 \right)^{m-j} \norm{\muh^{j-\hf}}{}^2 
	\\
&+ C \tau \sum_{m=0}^{\ell}\norm{\nabla \bbuh^{m+\hf}}{}^2 \cdot \sum_{j=1}^m \left( \frac13 \right)^{m-j} \norm{\muh^{j-\frac32}}{}^2 + C \tau \sum_{m=0}^{\ell}\norm{\nabla \bbuh^{m+\hf}}{}^2 
	\\
\le& \, C (T+1) + C \tau \sum_{m=0}^{\ell}\norm{\nabla \bbuh^{m+\hf}}{}^2 \cdot \sum_{j=1}^m \left( \frac13 \right)^{m-j} \left( \norm{\muh^{j-\hf}}{}^2 + \norm{\muh^{j-\frac32}}{}^2 \right) ,
	\end{align*}
for any $2 \le \ell \le M-1$, where we have used Part 1, \eqref{eq:sum-3D-good-mu} and \eqref{eq:sum-mu-discrete}. Moreover, with an application of the discrete Gronwall inequality from Lemma \ref{lem-discrete-gronwall-2} (with $\alpha = \frac13 < 1$), we arrive at 
	\begin{align} 
   \hf \norm{\muh^{\ell+\hf}}{}^2 + \frac{\tau}{16} \sum_{m=2}^{\ell} \norm{\dtau \phih^{m+\hf}}{}^2  \le & \ C (T+1) \cdot \exp \left( C A_\alpha \tau \sum_{m=0}^{\ell}\norm{\nabla \bbuh^{m+\hf}}{}^2 \right) 
	\nonumber
	\\
\le & \ C ( T+1) ,  
	\end{align} 
where \eqref{eq:sum-mu-discrete} has been repeatedly applied.

Now, using \eqref{eq:discrete-laplacian-phi-mu-equiv}, \eqref{eq:infinity-bound}, the embedding $H^1(\Omega) \hookrightarrow L^6(\Omega)$, and \eqref{eq:Linf-H1-phi-discrete}, we get 
	\begin{align*}
\norm{\Delta_h \phich^{\ell+\hf}}{}^2 + \norm{\phich^{\ell+\hf}}{L^\infty}^{\frac{4(6-d)}{d}} \le C (T+1) , \, \, \forall \, 2 \le \ell \le M-1 . 
	\end{align*}
By Lemma \ref{lem:induction}, the following bound is available: 
	\begin{align*}
\norm{\Delta_h \phih^{\ell+1}}{}^2 \le C (T+1) , \, \, \forall \, 2 \le \ell \le M-1 . 
	\end{align*}
Using \eqref{eq:infinity-bound} again, the embedding $H^1(\Omega) \hookrightarrow L^6(\Omega)$, and \eqref{eq:Linf-H1-phi-discrete}, we arrive at 
	\begin{align*}
 \norm{\phih^{\ell+1}}{L^\infty}^{\frac{4(6-d)}{d}} \le C (T+1) , \, \, \forall \, 2 \le \ell \le M-1  . 
	\end{align*}
The proof is completed by combining Parts 1 and 2.
	\end{proof}

	\section{Error Estimates for the Fully Discrete Scheme}
	\label{sec:error estimates}
	
	\begin{asmp}
	\label{asmp:initial-stability-3}
For the error estimates that we pursue in this section, we shall assume that weak solutions have the additional regularities
	\begin{align}
&\phi \in L^{\infty}(0,T;W^{1,6}(\Omega)) \cap H^1(0,T;H^{q+1}(\Omega)) \cap H^2(0,T;H^{3}(\Omega)) \cap H^3(0,T;L^2(\Omega)),
	\label{eq:higher-regularities-a}
	\\
&\phi^2 \in H^2(0,T;H^1(\Omega)),
	\label{eq:higher-regularities-b}
	\\
&\mu \in L^2(0,T;H^{q+1}(\Omega)),
	\label{eq:higher-regularities-c}
	\\
&\bu \in H^2(0,T;{\bf L}^2(\Omega)) \cap L^\infty(0,T;{\bf L}^4(\Omega)) \cap L^\infty(0,T;{\bf H}^{q+1}(\Omega)) \cap H^1(0,T;{\bf H}^{q+1}(\Omega)),
	\label{eq:higher-regularities-d}
	\\
& p \in L^2(0,T;H^q(\Omega) \cap L^2_0(\Omega)) \cap L^\infty(0,T;H^{q}(\Omega)) ,
	\label{eq:higher-regularities-e}
	\end{align}
where $q\ge 1$ corresponds to the finite element spaces defined at the beginning of Section \ref{sec:defn-and-properties}. The norm bounds associated with the assumed regularities above are not necessarily global-in-time and therefore can involve constants that depend upon the final time $T$.  We also assume that the initial data are sufficiently regular so that the stability from Assumptions \ref{asmp:initial-stability-1}, \ref{asmp:initial-stability-2}, and \ref{asmp:initial-stability-2-b} hold.
	\end{asmp}

Weak solutions $(\phi, \mu, \bu, p)$ to \eqref{eq:weak-chns-a} -- \eqref{eq:weak-chns-d} with the higher regularities \eqref{eq:higher-regularities-a} -- \eqref{eq:higher-regularities-e} solve the following variational problem: for all $t\in [0,T]$, 
		\begin{align}
\iprd{ \partial_t \phi}{\nu} + \varepsilon \,\aiprd{\mu}{\nu} + \bform{\phi}{\bu}{\nu} &= 0, \quad \forall \,\nu \in H^1(\Omega),
	\\
\iprd{\mu}{\psi}-\varepsilon \,\aiprd{\phi}{\psi} - \varepsilon^{-1}\iprd{\phi^3-\phi}{\psi}  &= 0, \quad \forall \,\psi\in H^1(\Omega),
	\\
\langle \partial_t \bu, \bv \rangle + \eta \, \aiprd{\bu}{\bv} + \Bform{\bu}{\bu}{\bv} - \cform{\bv}{p} - \gamma \, \bform{\phi}{\bv}{\mu} &=0, \quad \forall \, \bv \in {\bf H}^1_0(\Omega),
	\\
\cform{\bu}{q} &=0, \quad \forall \, q \in L^2_0(\Omega).
	\end{align}

We define the following: for any \emph{real} number $m\in [0,M]$, $t_{m} := m \cdot \tau$, and $\psi^{m} := \psi(t_{m})$. This definition applies to vector valued functions of time as well. Note that, in general,
	\[
\psi^{m+\hf}:= \psi(t_{m+\hf}) \ne \frac{1}{2}\left(\psi^{m}+\psi^{m+1}\right)=:\bar{\psi}^{m+\hf}.
	\]
An over-bar will always indicate a simple central average in time.
Denote
	\begin{equation}
	 \eA^{\phi,m} := \phi^{m} - R_h \phi^{m}, \quad \eA^{\mu,m} := \mu^{m} - R_h \mu^{m}, \quad  \eA^{\bu,m} := \bu^{m} - P_h \bu^{m}, \quad \eA^{p,m} := p^{m} - P_h p^{m}.
	\end{equation}
The following definitions are given for any integer $0\le m\le M-1$: 
	\begin{align*}
\dtau \phi^{m+\hf}  := & \ \frac{\phi^{m+1} - \phi^{m}}{\tau}, & \quad \dtau \bu^{m+\hf} := & \ \frac{\bu^{m+1} - \bu^{m}}{\tau}, 
	\\
\sigma^{\phi, m+\hf}_1 := & \ \dtau R_h \phi^{m+\hf} - \dtau \phi^{m+\hf}, & \quad \sigma^{\bu, m+\hf}_1 := & \ \dtau {\bf P}_h \bu^{m+\hf} - \dtau \bu^{m+\hf},
	\\
\sigma^{\phi, m+\hf}_2 := & \ \dtau \phi^{m+\hf} - \partial_t \phi^{m+\hf},  &\quad \sigma^{\bu, m+\hf}_2 := & \ \dtau \bu^{m+\hf} - \partial_t \bu^{m+\hf}, 
	\\
\sigma^{\phi, m+\hf}_3 := & \ \bar{\phi}^{m+\hf}  -  \phi^{m+\hf}, & \quad \sigma^{\bu, m+\hf}_3 := & \ \bar{\bu}^{m+\hf} -  \bu^{m+\hf},
	\\
\sigma_4^{\phi,m+\hf} := & \  \chi\left(\phi^{m+1},\phi^m \right) - \left(\phi^{m+\hf}\right)^3, &  \quad \sigma^{p, m+\hf}_3 := & \ \bar{p}^{m+\hf} -  p^{m+\hf}.
	\end{align*}	   
Then the PDE solution, evaluated at the half-integer time steps $t_{m+\hf}$, satisfies
	\begin{subequations}
	\begin{align}
\iprd{\dtau R_h \phi^{m+\hf}}{\nu} + \varepsilon \, \aiprd{R_h \mu^{m+\hf}}{\nu} = & \  \iprd{\sigma^{\phi, m+\hf}_1 + \sigma^{\phi, m+\hf}_2}{\nu} - \bform{\phi^{m+\hf}}{\bu^{m+\hf}}{\nu},
	\label{eq:weak-error-a}
	\\
\varepsilon \, \aiprd{ R_h \bar{\phi}^{m+\hf}}{\psi} - \iprd{R_h \mu^{m+\hf}}{\psi} = & \  \varepsilon \, \aiprd{\sigma^{\phi,m+\hf}_3}{\psi} + \iprd{\eA^{\mu,m+\hf}}{\psi} 
	\nonumber
	\\
&- \frac{1}{\varepsilon} \iprd{\chi\left(\phi^{m+1},\phi^m\right)}{\psi} + \frac{1}{\varepsilon}\iprd{\sigma_4^{\phi,m+\hf}}{\psi}
	\nonumber
	\\
& + \frac{1}{\varepsilon} \iprd{\phi^{m+\hf}}{\psi} ,
	\label{eq:weak-error-b}
	\\
\iprd{\dtau {\bf P}_h \bu^{m+\hf}}{\bv} + \eta \, \aiprd{{\bf P}_h\bar{\bu}^{m+\hf} }{\bv} & 
	\nonumber
	\\
- \cform{\bv}{P_h \bar{p}^{m+\hf}}  = & \ \eta \, \aiprd{\sigma^{\bu,m+\hf}_3}{\bv} + \iprd{\sigma^{\bu, m+\hf}_1 + \sigma^{\bu,m+\hf}_2}{\bv}  - \cform{\bv}{\sigma^{p,m+\hf}_3} 
	\nonumber
	\\
& + \gamma \, \bform{\phi^{m+\hf}}{\bv}{\mu^{m+\hf}} - \Bform{\bu^{m+\hf}}{\bu^{m+\hf}}{\bv} ,
	\label{eq:weak-error-c}
	\\
\cform{ {\bf P}_h \bar{\bu}^{m+\hf}}{q} = & \ \cform{\sigma^{\bu,m+\hf}_3}{q},
	\label{eq:weak-error-d}
	\end{align}
	\end{subequations}
for all $\nu, \psi \in S_h$, $\bv \in \Xh$, and $q \in \Soh$, for any $0\le m\le M-1$.

Restating the fully discrete splitting scheme \eqref{eq:scheme-a} -- \eqref{eq:scheme-d}, we have, for $1 \le m \le M-1$, and for all $\nu, \psi \in S_h$, $\bv \in \Xh$, and $q \in \Soh$, 
	\begin{subequations}
	\begin{align}
\iprd{\dtau \phih^{m+\hf}}{\nu} + \varepsilon \, \aiprd{\muh^{m+\hf}}{\nu} = & \ - \bform{\tilde{\phi}_h^{m+\hf}}{\bbuh^{m+\hf}}{\nu} ,
	\label{eq:scheme-error-a}
	\\
\varepsilon \, \aiprd{\bphih^{m+\hf}}{\psi} + \frac{\varepsilon}{4} \, \aiprd{\tau^2 \ddtau \phih^m}{\psi} - \iprd{\muh^{m+\hf}}{\psi} = & 
  - \frac{1}{\varepsilon} \iprd{\chi\left(\phih^{m+1}, \phih^m\right)}{\psi} + \frac{1}{\varepsilon} \iprd{\tilde{\phi}_h^{m+\hf}}{\psi}  ,
	\label{eq:scheme-error-b}
	\\
\iprd{\dtau \buh^{m+\hf}}{\bv} + \eta \, \aiprd{\bbuh^{m+\hf}}{\bv}  - \cform{\bv}{\bar{p}_h^{m+\hf}}  = & \  \gamma \, \bform{\tilde{\phi}_h^{m+\hf}}{\bv}{\muh^{m+\hf}} - \Bform{\tilde{\bu}_h^{m+\hf}}{\bbuh^{m+\hf}}{\bv} ,
	\label{eq:scheme-error-c}
	\\
\cform{\bbuh^{m+\hf}}{q} = & \ 0 ,
	\label{eq:scheme-error-d}
	\end{align}
	\end{subequations}
where $\ddtau\psi^m  := \frac{1}{\tau^2}\left( \psi^{m+1}-2\psi^m+\psi^{m-1}\right)$.
	
Now let us define the following additional error terms
	\begin{align}
	\nonumber
\eh^{\phi,m} := R_h \phi^{m} - \phih^{m}, \quad \e^{\phi,m} := \phi^{m} - \phih^{m}, \quad \eh^{\mu,m} := R_h \mu^{m} - \muh^{m},
	\\
\eh^{\bu,m} := P_h \bu^{m} - \buh^{m}, \quad \e^{\bu,m} := \bu^{m} - \buh^{m}, \quad \eh^{p,m} := P_h p^{m} - p_h^{m}, \quad \e^{p,m} := p^{m} - p_h^{m}. 
	\end{align}
We also define, for $1\le m\le M-1$,
	\begin{align}
\sigma_5^{\phi,m+\hf} &:= \, \chi\left(\phih^{m+1},\phih^m \right) - \chi\left(\phi^{m+1},\phi^m \right)  ,
	\\
\sigma_6^{\phi,m+\hf}&:= \, \phi^{m+\hf} - \tilde{\phi}_h^{m+\hf} ,
	\\
\sigma_6^{\bu,m+\hf}&:= \, \bu^{m+\hf} - \tilde{\bu}_h^{m+\hf} 	.
	\end{align}

Subtracting \eqref{eq:scheme-error-a} -- \eqref{eq:scheme-error-d} from \eqref{eq:weak-error-a} -- \eqref{eq:weak-error-d}, yields, for $1\le m\le M-1$,
	\begin{subequations}
	\begin{align}
\iprd{\dtau \eh^{\phi,m+\hf}}{\nu} + \varepsilon \, \aiprd{\eh^{\mu,m+\hf}}{\nu}  = & \ \iprd{\sigma^{\phi,m+\hf}_1 + \sigma^{\phi,m+\hf}_2}{\nu} - \bform{\phi^{m+\hf}}{\bu^{m+\hf}}{\nu} 
	\nonumber
	\\
&+ \bform{\tilde{\phi}_h^{m+\hf}}{\bbuh^{m+\hf}}{\nu},
	\label{eq:error-a}
	\\
 \varepsilon  \, \aiprd{\bar{\e}_h^{\phi,m+\hf} }{\psi} + \frac{\varepsilon\tau^2}{4} \aiprd{\ddtau \eh^{\phi,m}}{\psi} \hspace{0.2in} &
 	\nonumber
	\\
- \iprd{\eh^{\mu,m+\hf}}{\psi} = & \ \varepsilon \, \aiprd{\sigma_3^{\phi,m+\hf}}{\psi} + \iprd{\eA^{\mu,m+\hf}}{\psi} 
	\nonumber
	\\
&+ \frac{1}{\varepsilon} \iprd{\sigma_5^{\phi,m+\hf}}{\psi} + \frac{1}{\varepsilon}\iprd{\sigma_4^{\phi,m+\hf}}{\psi} 
	\nonumber
	\\
&+ \frac{1}{\varepsilon} \iprd{ \sigma_6^{\phi,m+\hf}}{\psi} + \frac{\varepsilon\tau^2}{4} \, \aiprd{\ddtau \phi^m}{\psi} ,
	\label{eq:error-b}
	\\
\iprd{\dtau \eh^{\bu,m+\hf}}{\bv} + \eta \, \aiprd{\bar{\e}_h^{\bu,m+\hf}}{\bv} \hspace{0.2in}  &
	\nonumber
	\\
- \cform{\bv}{\bar{\e}_h^{p,m+\hf}} = & \ \iprd{\sigma^{\bu,m+\hf}_1 + \sigma^{\bu,m+\hf}_2}{\bv} + \eta \, \aiprd{\sigma_3^{\bu,m+\hf}}{\bv}  
	\nonumber
	\\
&  - \cform{\bv}{\sigma_3^{p,m+\hf}} +\gamma \, \bform{\phi^{m+\hf}}{\bv}{\mu^{m+\hf}}  
	\nonumber
	\\
& - \Bform{\bu^{m+\hf}}{\bu^{m+\hf}}{\bv} - \gamma \, \bform{\tilde{\phi}_h^{m+\hf}}{\bv}{\muh^{m+\hf}} 
	\nonumber
	\\
&+ \Bform{\tilde{\bu}_h^{m+\hf}}{\bbuh^{m+\hf}}{\bv} ,
	\label{eq:error-c}
	\\
 \cform{\bar{\e}_h^{\bu,m+\hf} }{q} = & \ \cform{\sigma_3^{\bu,m+\hf}}{q} (= 0),
	\label{eq:error-d}
	\end{align}
	\end{subequations}

Setting $\nu = \eh^{\mu, m+\hf}$ in \eqref{eq:error-a}, $\psi = \dtau \eh^{\phi, m+\hf}$ in \eqref{eq:error-b}, $\bv = \frac{1}{\gamma} \bar{\e}_h^{\bu,m+\hf}$ in \eqref{eq:error-c}, $q = \frac{1}{\gamma}\bar{\e}_h^{p,m+\hf}$ in \eqref{eq:error-d} and adding the resulting equations, we have
	\begin{align}
& \hspace{-0.30in}\frac{\varepsilon}{2\tau} \left( \norm{\nabla\eh^{\phi,m+1}}{}^2 - \norm{\nabla \eh^{\phi,m}}{}^2\right) + \frac{1}{2 \gamma \tau} \left(\norm{ \eh^{\bu,m+1}}{}^2 - \norm{\eh^{\bu,m}}{}^2\right)  
	\nonumber
	\\
& \hspace{-0.25in} + \varepsilon \norm{\nabla \eh^{\mu,m+\hf}}{}^2 + \frac{ \varepsilon\tau^2}{4} \, \aiprd{\ddtau\eh^{\phi,m}}{\dtau \eh^{\phi, m+\hf}} + \frac{\eta}{\gamma} \, \norm{ \nabla \bar{\e}_h^{\bu,m+\hf}}{}^2  
	\nonumber
	\\
& = \iprd{\sigma^{\phi,m+\hf}_1 + \sigma^{\phi,m+\hf}_2}{\eh^{\mu,m+\hf}} + \varepsilon \, \aiprd{\sigma_3^{\phi,m+\hf}}{\dtau \eh^{\phi, m+\hf}} + \iprd{\eA^{\mu,m+\hf}}{\dtau \eh^{\phi,m+\hf}}
	\nonumber
	\\
&\quad + \frac{1}{\varepsilon} \iprd{\sigma_4^{\phi,m+\hf} + \sigma_5^{\phi,m+\hf} + \sigma_6^{\phi,m+\hf} }{\dtau \eh^{\phi, m+\hf}} + \frac{\varepsilon\tau^2}{4} \, \aiprd{\ddtau\phi^m }{\dtau \eh^{\phi, m+\hf}} 
	\nonumber
	\\
&\quad  + \frac{1}{\gamma} \iprd{\sigma^{\bu,m+\hf}_1 + \sigma^{\bu,m+\hf}_2}{\bar{\e}_h^{\bu,m+\hf}} + \frac{\eta}{\gamma} \, \aiprd{\sigma_3^{\bu,m+\hf}}{\bar{\e}_h^{\bu,m+\hf} } - \frac{1}{\gamma} \cform{\bar{\e}_h^{\bu,m+\hf} }{ \sigma_3^{p,m+\hf}} 
	\nonumber
	\\
&\quad  - \bform{\phi^{m+\hf}}{\bu^{m+\hf}}{\eh^{\mu,m+\hf}} + \bform{\tilde{\phi}_h^{m+\hf}}{\bbuh^{m+\hf}}{\eh^{\mu,m+\hf}} 
	\nonumber
	\\
&\quad + \bform{\phi^{m+\hf}}{\bar{\e}_h^{\bu,m+\hf}}{\mu^{m+\hf}} - \bform{\tilde{\phi}_h^{m+\hf}}{\bar{\e}_h^{\bu,m+\hf}}{\muh^{m+\hf}} 
	\nonumber
	\\
&\quad- \frac{1}{\gamma} \Bform{\bu^{m+\hf}}{\bu^{m+\hf}}{\bar{\e}_h^{\bu,m+\hf}} + \frac{1}{\gamma} \Bform{\tilde{\bu}_h^{m+\hf}}{\bbuh^{m+\hf}}{\bar{\e}_h^{\bu,m+\hf}} ,
	\label{eq:error-eq}
	\end{align}
for all $1\le m\le M-1$. Expression~\eqref{eq:error-eq} is the key error equation from which we will define our error estimates. Observe that the error equation is not defined for $m=0$.

The following estimates are standard and the proofs are omitted.
	\begin{lem}
	\label{lem:phi-truncation-errors}
Suppose that $(\phi, \mu, \bu, p)$ is a weak solution to \eqref{eq:weak-error-a} -- \eqref{eq:weak-error-d}, with the additional regularities in Assumption~\ref{asmp:initial-stability-3}.  Then for all $t_m\in[0,T]$ and for any $h$, $\tau >0$, there exists a constant $C>0$, independent of $h$ and $\tau$ and $T$, such that for all $0\le m\le M-1$,
	\begin{align}
\norm{\sigma^{\phi,m+\hf}_1}{}^2 &\le  C\frac{h^{2q+2}}{\tau} \int_{t_m}^{t_{m+1}} \norm{\partial_s\phi(s)}{H^{q+1}}^2  ds , 
	\label{eq:phi-truncation-1} 
	\\
\norm{\sigma^{\phi,m+\hf}_2}{}^2 &\le \frac{\tau^3}{640} \int_{t_m}^{t_{m+1}}\norm{\partial_{sss}\phi(s)}{}^2 ds,   
	\label{eq:phi-truncation-2}
	\\
\norm{\nabla \Delta \sigma^{\phi,m+\hf}_3}{}^2 &\le \frac{\tau^3}{96} \int_{t_m}^{t_{m+1}} \norm{\nabla \Delta \partial_{ss} \phi(s)}{}^2 ds,   
	\label{eq:phi-truncation-3}
	\\
\norm{\nabla \sigma^{\phi,m+\hf}_3}{}^2 &\le \frac{\tau^3}{96} \int_{t_m}^{t_{m+1}} \norm{\nabla \partial_{ss} \phi(s)}{}^2 ds,   
	\label{eq:phi-truncation-3b}
	\\
\norm{ \hf\left(\phi^{m+1}\right)^2+\hf\left(\phi^m\right)^2 - \left(\phi^{m+\hf}\right)^2}{H^1}^2 &\le \frac{\tau^3}{96} \int_{t_m}^{t_{m+1}} \norm{\partial_{ss} \phi^2(s)}{H^1}^2 ds.
	\label{eq:phi-truncation-4}
	\end{align}
In addition, for all $1\le m\le M-1$,
	\begin{align}
\norm{\tau^2\nabla \Delta \ddtau\phi^m}{}^2 &\le \frac{\tau^3}{3} \int_{t_{m-1}}^{t_{m+1}} \norm{\nabla \Delta \partial_{ss} \phi(s)}{}^2 ds,  
	\label{eq:phi-truncation-5}
	\\
\norm{\tau^2\nabla \ddtau\phi^m}{}^2 &\le \frac{\tau^3}{3} \int_{t_{m-1}}^{t_{m+1}} \norm{\nabla \partial_{ss} \phi(s)}{}^2 ds,   
	\label{eq:phi-truncation-6}
	\\
\norm{\nabla\sigma_6^{\phi,m+\hf}}{}^2 &\le  \frac{\tau^3}{12}  \int_{t_{m-1}}^{t_{m+1}}\norm{\nabla\partial_{ss}\phi(s)}{}^2 \, ds.
	\label{eq:phi-truncation-7}
	\end{align}
	\end{lem}
	
	\begin{lem}
	\label{lem:u-truncation-errors}
Suppose that $(\phi, \mu, \bu, p)$ is a weak solution to \eqref{eq:weak-error-a} -- \eqref{eq:weak-error-d}, with the additional regularities in Assumption~\ref{asmp:initial-stability-3}.  Then for all $t_m\in[0,T]$ and for any $h$, $\tau >0$, there exists a constant $C>0$, independent of $h$ and $\tau$ and $T$, such that for all $0\le m\le M-1$,
	\begin{align}
\norm{\sigma^{\bu,m+\hf}_1}{}^2 &\le  C\frac{h^{2q+2}}{\tau} \int_{t_m}^{t_{m+1}} \norm{\partial_s\bu(s)}{H^{q+1}}^2  ds , 
	\label{eq:u-truncation-1} 
	\\
\norm{\sigma^{\bu,m+\hf}_2}{}^2 &\le \frac{\tau^3}{640} \int_{t_m}^{t_{m+1}}\norm{\partial_{sss}\bu(s)}{}^2 ds,   
	\label{eq:u-truncation-2}
	\\
\norm{\nabla \Delta \sigma^{\bu,m+\hf}_3}{}^2 &\le \frac{\tau^3}{96} \int_{t_m}^{t_{m+1}} \norm{\nabla \Delta \partial_{ss} \bu(s)}{}^2 ds,   
	\label{eq:u-truncation-3}
	\\
\norm{\nabla \sigma^{\bu,m+\hf}_3}{}^2 &\le \frac{\tau^3}{96} \int_{t_m}^{t_{m+1}} \norm{\nabla \partial_{ss} \bu(s)}{}^2 ds,
	\label{eq:u-truncation-3b}
	\\
\norm{\sigma^{p,m+\hf}_3}{}^2 &\le \frac{\tau^3}{96} \int_{t_m}^{t_{m+1}} \norm{\partial_{ss} p(s)}{}^2 ds.  
	\label{eq:p-truncation-3}
	\end{align}
In addition, for all $1\le m\le M-1$,
	\begin{align}
\norm{\tau^2\nabla \Delta \ddtau\bu^m}{}^2 &\le \frac{\tau^3}{3} \int_{t_{m-1}}^{t_{m+1}} \norm{\nabla \Delta \partial_{ss} \bu(s)}{}^2 ds,  
	\label{eq:u-truncation-5}
	\\
\norm{\tau^2\nabla \ddtau\bu^m}{}^2 &\le \frac{\tau^3}{3} \int_{t_{m-1}}^{t_{m+1}} \norm{\nabla \partial_{ss} \bu(s)}{}^2 ds,   
	\label{eq:u-truncation-6}
	\\
\norm{\nabla\sigma_6^{\bu,m+\hf}}{}^2 &\le  \frac{\tau^3}{12}  \int_{t_{m-1}}^{t_{m+1}}\norm{\nabla\partial_{ss}\bu(s)}{}^2 \, ds,   
	\label{eq:u-truncation-7}
	\\
\norm{\tau^2\nabla \ddtau p^m}{}^2 &\le \frac{\tau^3}{3} \int_{t_{m-1}}^{t_{m+1}} \norm{\nabla \partial_{ss} p(s)}{}^2 ds .
	\label{eq:p-truncation}
	\end{align}
	\end{lem}

The following estimates are proved in~\cite{diegel15b}.
	\begin{lem}
	\label{lem:phi-sigma4-estimate}
Suppose that $(\phi, \mu, \bu, p)$ is a weak solution to \eqref{eq:weak-error-a} -- \eqref{eq:weak-error-d}, with the additional regularities in Assumption~\ref{asmp:initial-stability-3}. Then, there exists a constant $C>0$ independent of $h$ and $\tau$ -- but possibly dependent upon $T$ through the regularity estimates -- such that, for any $h, \tau >0$,
	\begin{align}
\norm{\nabla \sigma_4^{\phi,m+\hf}}{}^2 \le& \, C\tau^3 \int_{t_m}^{t_{m+1}} \norm{\nabla \partial_{ss} \phi(s)}{}^2 ds +C \tau^3 \int_{t_m}^{t_{m+1}} \norm{\partial_{ss} \phi^2(s)}{H^1}^2 ds,
	\\
\norm{\nabla \sigma_5^{\phi,m+\hf}}{}^2 \le& \, C \norm{\nabla \e^{\phi,m+1}}{}^2 + C \norm{\nabla \e^{\phi,m}}{}^2,
	\end{align}
where $\e^{\phi,m} := \phi^{m} - \phih^{m}$.
	\end{lem}

	\begin{lem}
	\label{lem:sigma6-estimate}
Suppose that $(\phi, \mu, \bu, p)$ is a weak solution to \eqref{eq:weak-error-a} -- \eqref{eq:weak-error-d}, with the additional regularities in Assumption~\ref{asmp:initial-stability-3}. Then, there exists a constant $C>0$ independent of $h$ and $\tau$ such that, for any $h, \tau >0$,
	\begin{align}
\norm{\nabla \sigma_6^{\phi,m+\hf}}{}^2 \le & \ C \tau^3   \int_{t_{m-1}}^{t_{m+1}}\norm{\nabla\partial_{ss}\phi(s)}{}^2 \, ds + C\norm{\nabla \e^{\phi,m}}{}^2+ C\norm{\nabla \e^{\phi,m-1}}{}^2,
	\label{eq:phi-sigma6}
	\\
\norm{\nabla \sigma_6^{\bu,m+\hf}}{}^2 \le& \ C \tau^3   \int_{t_{m-1}}^{t_{m+1}}\norm{\nabla\partial_{ss}\bu(s)}{}^2 \, ds + C\norm{\nabla \e^{\bu,m}}{}^2+ C\norm{\nabla \e^{\bu,m-1}}{}^2 ,
	\label{eq:u-sigma6}
	\end{align}
where $\e^{\phi,m} := \phi^{m} - \phih^{m}$ and $\e^{\bu,m} := \bu^m - \buh^m$.
	\end{lem}
	\begin{proof}
For $1\le m\le M-1$, using the truncation error estimate~\eqref{eq:phi-truncation-7}, we obtain
	\begin{equation}
\norm{\nabla\sigma_6^{\phi,m+\hf}}{}^2 \le 3\frac{\tau^3}{12}  \int_{t_{m-1}}^{t_{m+1}}\norm{\nabla\partial_{ss}\phi(s)}{}^2 \, ds + \frac{27}{4}\norm{\nabla\e^{\phi,m}}{}^2 + \frac{3}{4}\norm{\nabla\e^{\phi,m-1}}{}^2 .
	\end{equation}
Estimate \eqref{eq:u-sigma6} similarly follows.
	\end{proof}

The following technical lemma is proved in \cite{diegel15}.
	\begin{lem}
	\label{lem:technical}
Suppose $g \in H^1(\Omega)$, and $v \in \Soh$.  Then
	\begin{equation}
\left|\iprd{g}{v}\right| \le C \norm{\nabla g}{} \, \norm{v}{-1,h}  ,
	\end{equation}
for some $C>0$ that is independent of $h$.
	\end{lem}
	
We use only some very basic estimates for the trilinear form $B$:
	\begin{lem}
	\label{lem:technical-Bform}
Suppose $\bu, \bv, \bw \in {\bf H}_0^1(\Omega)$. Then
	\begin{equation}
\left|\Bform{\bu}{\bv}{\bw}\right| \le C \norm{\nabla \bu}{} \norm{\nabla \bv}{} \norm{\nabla \bw}{}.
	\label{lem:technical-Bform-1}
	\end{equation}
If $\bu\in{\bf L}^\infty(\Omega)$ and  $\bv, \bw \in {\bf H}_0^1(\Omega)$, then
	\begin{equation}
\left|\Bform{\bu}{\bv}{\bw}\right| \le C \norm{ \bu}{L^\infty} \norm{\nabla \bv}{} \norm{\nabla \bw}{}.	
	\label{lem:technical-Bform-3}
	\end{equation}
If $\bu \in {\bf L}^2(\Omega)$, $\bv,\bw\in {\bf H}^1(\Omega)\cap{\bf L}^\infty(\Omega)$, then
	\begin{equation}
\left|\Bform{\bu}{\bv}{\bw}\right| \le  \norm{\bu}{} \left(\norm{\nabla \bv}{} \norm{\bw}{L^{\infty}} + \norm{\nabla \bw}{} \norm{\bv}{L^{\infty}}\right).
	\label{lem:technical-Bform-5}
	\end{equation}
	\end{lem}
	
We also recall some basic inverse inequalities.
    \begin{align}
\norm{\varphi_h}{W^m_q} \leq Ch^{\nicefrac{d}{q} - \nicefrac{d}{p}}h^{\ell-m}\norm{\varphi_h}{W^\ell_p}, &\quad \forall\, \varphi_h \in \mathcal{M}_r^h, \ 1 \le p \le q \le \infty, \ 0\le \ell \le m \le 1,
	\label{Inverse-estimate-1}
	\end{align}
From this and the Gagliardo-Nirenburg and Poincar\'{e} inequalities it follows that~\cite{baker76}
	\begin{equation}
\norm{\varphi_h}{L^\infty} \le C h^{\frac{1}{2} -\frac{d}{2}}\norm{\varphi_h}{ }^\frac{1}{2}\norm{\nabla\varphi_h}{ }^\frac{1}{2} , \quad d= 2,3,
	\label{sobolev-inverse-estimate}
	\end{equation}
for all $\varphi_h\in \mathcal{M}_{r,0}^h$.

	\begin{lem}
Let $( {\bf P}_h, P_h) : {\bf V} \times L^2_0 \to \Vh \times \Soh$ be defined as in \eqref{eq:Stokes-projection} and suppose that $(\phi, \mu, \bu, p)$ is a weak solution to \eqref{eq:weak-error-a} -- \eqref{eq:weak-error-d}, with the additional regularities in Assumption~\ref{asmp:initial-stability-3}. Then, for any $h$, $\tau >0$ there exists a constant $C>0$, independent of $h$ and $\tau$, such that, for $0 \le m \le M-1$,
	\begin{equation}
\norm{{\bf P}_h \bu}{L^\infty(0,T;L^{\infty}(\Omega))} \le C,
	\label{eq:max-stability-velecity-projection}
	\end{equation}
and, as a simple consequence, 
	\begin{equation}
\norm{\eA^{\bu}}{L^\infty(0,T;L^{\infty}(\Omega))} \le C.
	\label{eq:max-stability-velecity-approx-error}
	\end{equation}
	\end{lem}
	
	\begin{proof}
Let $\bw =\mathcal{I}_h \bu\in \Xh$, the standard Lagrange nodal interpolant of $\bu$. Following Baker's unpublished paper \cite{baker76} and using standard finite element approximations, including \eqref{stokes-estimate}, inverse inequalities, and Sobolev's embedding theorem, we have
	\begin{align*}
\norm{{\bf P}_h \bu}{L^\infty} &= \norm{{\bf P}_h \bu - \bw + \bw - \bu + \bu}{L^\infty}
	\\
&\le \norm{{\bf P}_h \bu - \bw}{L^\infty} + \norm{\bw - \bu}{L^\infty} + \norm{\bu}{L^\infty}
	\\
&\le C h^{-\frac{d}{2}} \norm{{\bf P}_h \bu - \bw}{} + \norm{\bw - \bu}{L^\infty} + \norm{\bu}{L^\infty}
	\\
&\le C h^{-\frac{d}{2}} \left(\norm{{\bf P}_h \bu - \bu}{} + \norm{\bu - \bw}{}\right) + \norm{\bw - \bu}{L^\infty} + \norm{\bu}{L^\infty}
	\\
&\le C \left( \norm{\bu - \bw}{L^\infty} + h^{-\frac{d}{2}}\norm{\bu - \bw}{}\right) + C h^{-\frac{d}{2}} \norm{{\bf P}_h \bu - \bu}{} + \norm{\bu}{L^\infty}
	\\
&\le \norm{\bu}{L^\infty} + C h^{q+1-\frac{d}{2}} \left(\left| \bu \right|_{H^{q+1}} + \left| p \right|_{H^{q}} \right).
	\end{align*}
Taking the $L^\infty$ norm over $(0,T)$ and noting that $q \ge 1$, the proof is concluded.
	\end{proof}
	
We now proceed to estimate the terms on the right-hand-side of \eqref{eq:error-eq}. 

	\begin{lem}
	\label{lem:error-rhs-control}
Suppose that $(\phi, \mu, \bu, p)$ is a weak solution to \eqref{eq:weak-error-a} -- \eqref{eq:weak-error-d}, with the additional regularities in Assumption~\ref{asmp:initial-stability-3}.  Then, for any $h$, $\tau >0$ and any $\alpha > 0$ there exist a constant $C=C(\alpha,T)>0$, independent of $h$ and $\tau$, such that, for $1 \le m \le M-1$,
	\begin{align}
& \frac{\varepsilon}{2\tau} \left( \norm{\nabla\eh^{\phi,m+1}}{}^2 - \norm{\nabla \eh^{\phi,m}}{}^2\right) + \frac{1}{2 \tau \gamma} \left(\norm{ \eh^{\bu,m+1}}{}^2 - \norm{\eh^{\bu,m}}{}^2\right) + \frac{\varepsilon\tau^2}{4} \, \aiprd{\ddtau\eh^{\phi,m}}{\dtau \eh^{\phi, m+\hf}} 
	\nonumber
	\\
& \qquad + \frac{3\varepsilon}{4} \norm{\nabla \eh^{\mu,m+\hf}}{}^2 + \frac{\eta}{2\gamma} \, \norm{\nabla \bar{\e}_h^{\bu,m+\hf} }{}^2 \le \ C \norm{\nabla \eh^{\phi,m+1}}{}^2 + C \norm{\nabla \eh^{\phi,m}}{}^2  +  C \norm{\nabla \eh^{\phi,m-1}}{}^2  
	\nonumber
	\\
& \qquad +  C \norm{\eh^{\bu,m}}{}^2 +  C \norm{\eh^{\bu,m-1}}{}^2 + \alpha \norm{\dtau \eh^{\phi,m+\hf}}{-1,h}^2  + C \mathcal{R}^{m+\hf},
	\label{eq:right-hand-side-estimate}
	\end{align}
where	
	\begin{align}
\mathcal{R}^{m+\hf} := &\,  \frac{h^{2q+2}}{\tau} \int_{t_m}^{t_{m+1}} \norm{\partial_s\phi(s)}{H^{q+1}}^2  ds + \frac{h^{2q+2}}{\tau} \int_{t_m}^{t_{m+1}} \norm{\partial_s\bu(s)}{H^{q+1}}^2  ds +  \tau^3 \int_{t_m}^{t_{m+1}}\norm{\partial_{sss}\phi(s)}{}^2 ds 
	\nonumber
	\\
&+  \tau^3 \int_{t_m}^{t_{m+1}}\norm{\partial_{sss}\bu(s)}{}^2 ds +  \tau^3 \int_{t_m}^{t_{m+1}} \norm{\partial_{ss} \phi^2(s)}{H^1}^2 ds +  \tau^3 \int_{t_{m-1}}^{t_{m+1}}\norm{\nabla\partial_{ss}\phi(s)}{}^2 \, ds 
	\nonumber
	\\
& +  \tau^3 \int_{t_{m-1}}^{t_{m+1}}\norm{\nabla\partial_{ss}\bu(s)}{}^2 \, ds +  \tau^3 \int_{t_{m-1}}^{t_{m+1}} \norm{\nabla \Delta \partial_{ss} \phi(s)}{}^2 ds + \tau^3 \int_{t_m}^{t_{m+1}} \norm{\partial_{ss} p(s)}{}^2 ds
	\nonumber
	\\
& + h^{2q}\left| \mu^{m+\hf} \right|_{H^{q+1}}^2 +  h^{2q}\left| \phi^{m+1} \right|_{H^{q+1}}^2  +  h^{2q}\left| \phi^{m} \right|_{H^{q+1}}^2 +  h^{2q}\left| \phi^{m-1} \right|_{H^{q+1}}^2 +  h^{2q}\left| \bu^{m+1} \right|_{H^{q+1}}^2  
	\nonumber
	\\
& +  h^{2q}\left| \bu^{m} \right|_{H^{q+1}}^2 +  h^{2q}\left| \bu^{m-1} \right|_{H^{q+1}}^2 +  h^{2q}\left| p^{m+1} \right|_{H^{q+1}}^2  +  h^{2q}\left| p^{m} \right|_{H^{q+1}}^2  +  h^{2q}\left| p^{m-1} \right|_{H^{q+1}}^2.
	\label{eq:consistency}
	\end{align}

	\end{lem}

	\begin{proof}
Define, for $1\le m\le M-1$, time-dependent spatial mass average
	\begin{equation}
\overline{\e_h^{\mu,m+\hf}} := |\Omega|^{-1}\iprd{\e_h^{\mu,m+\hf}}{1}.
	\end{equation}
Using the Cauchy-Schwarz inequality, the Poincar\'{e} inequality, with the fact that 
	\[
\iprd{\sigma^{\phi,m+\hf}_1 + \sigma^{\phi,m+\hf}_2}{1} = 0,
	\]
and the local truncation error estimates~\eqref{eq:phi-truncation-1} and \eqref{eq:phi-truncation-2}, we get the following estimate:
	\begin{align}
& \hspace{-0.5in}\left|\iprd{\sigma^{\phi,m+\hf}_1 + \sigma^{\phi,m+\hf}_2}{\eh^{\mu,m+\hf}}\right| = \left|\iprd{\sigma^{\phi,m+\hf}_1 + \sigma^{\phi,m+\hf}_2}{\eh^{\mu,m+\hf}-\overline{\e_h^{\mu,m+\hf}}}\right|
	\nonumber
	\\
\le & \, \norm{\sigma^{\phi,m+\hf}_1 + \sigma^{\phi,m+\hf}_2}{} \cdot \norm{\eh^{\mu,m+\hf}-\overline{\e_h^{\mu,m+\hf}}}{}
	\nonumber
	\\
\le& \, C \norm{\sigma^{\phi,m+\hf}_1 + \sigma^{\phi,m+\hf}_2}{} \cdot \norm{\nabla \eh^{\mu,m+\hf}}{}
	\nonumber
	\\
\le& \, C\norm{\sigma^{\phi,m+\hf}_1}{}^2 + C\norm{\sigma^{\phi,m+\hf}_2}{}^2 + \frac{\varepsilon}{8} \norm{\nabla \eh^{\mu,m+\hf}}{}^2
	\nonumber
	\\
\le& \, C  \frac{h^{2q+2}}{\tau} \int_{t_m}^{t_{m+1}} \norm{\partial_s\phi(s)}{H^{q+1}}^2  ds
 + C\frac{\tau^3}{640} \int_{t_m}^{t_{m+1}}\norm{\partial_{sss}\phi(s)}{}^2 ds   + \frac{\varepsilon}{8} \norm{\nabla \eh^{\mu,m+\hf}}{}^2.
	\label{eq:error-estimate-1}
	\end{align}
Meanwhile, standard finite element approximation theory shows that
	\begin{equation*}
\norm{\nabla \eA^{\mu,m+\hf}}{} = \norm{\nabla \left(R_h \mu^{m+\hf} - \mu^{m+\hf}\right)}{} \leq C h^q\left| \mu^{m+\hf} \right|_{H^{q+1}} .
	\end{equation*}
Applying Lemma~\ref{lem:technical} and the last estimate, we have
	\begin{align}
\left|\iprd{\eA^{\mu,m+\hf}}{\dtau \eh^{\phi,m+\hf}}\right| \le& \, C \norm{\nabla\eA^{\mu,m+\hf}}{} \, \norm{\dtau \eh^{\phi,m+\hf}}{-1,h} 
	\nonumber
	\\
\le& \, C h^q\left| \mu^{m+\hf} \right|_{H^{q+1}} \norm{\dtau \eh^{\phi,m+\hf}}{-1,h}
	\nonumber
	\\
\le& \, Ch^{2q}\left| \mu^{m+\hf} \right|_{H^{q+1}}^2 + \frac{\alpha}{6} \norm{\dtau \eh^{\phi,m+\hf}}{-1,h}^2 .
	\label{eq:error-estimate-2}
	\end{align}
Using Lemma~\ref{lem:technical} and estimate \eqref{eq:phi-truncation-3}, we find
	\begin{align}
\varepsilon \, \aiprd{\sigma_3^{\phi,m+\hf}}{\dtau \eh^{\phi,m+\hf}} &= -\varepsilon \, \iprd{\Delta \sigma_3^{\phi,m+\hf}}{\dtau \eh^{\phi,m+\hf}}
	\nonumber
	\\
&\le C \norm{\nabla \Delta \sigma_3^{\phi,m+\hf}}{} \norm{\dtau \eh^{\phi,m+\hf}}{-1,h}
	\nonumber
	\\
&\le C \, \frac{\tau^3}{96} \int_{t_m}^{t_{m+1}} \norm{\nabla \Delta \partial_{ss} \phi(s)}{}^2 ds + \frac{\alpha}{6} \norm{\dtau \eh^{\phi,m+\hf}}{-1,h}^2.
	\nonumber
	\\
& \mbox{}
	\label{eq:error-estimate-3}
	\end{align}
	
Now, using Lemmas~\ref{lem:phi-sigma4-estimate} and \ref{lem:technical}, we obtain
	\begin{align}
\varepsilon^{-1}\left|\iprd{\sigma_4^{\phi,m+\hf}}{\dtau \eh^{\phi,m+\hf}}\right| \le& \, C \norm{\nabla \sigma_4^{\phi,m+\hf}}{} \, \norm{\dtau \eh^{\phi,m+\hf}}{-1,h}
	\nonumber
	\\
\le& \, C \norm{\nabla \sigma_4^{\phi,m+\hf}}{}^2 +  \frac{\alpha}{6}  \norm{\dtau \eh^{\phi,m+\hf}}{-1,h}^2
	\nonumber
	\\
\le& \, C\tau^3 \int_{t_m}^{t_{m+1}} \norm{\nabla \partial_{ss} \phi(s)}{}^2 ds 
	\nonumber
	\\
& + C \tau^3 \int_{t_m}^{t_{m+1}} \norm{\partial_{ss} \phi^2(s)}{H^1}^2 ds + \frac{\alpha}{6}  \norm{\dtau \eh^{\phi,m+\hf}}{-1,h}^2.
	\label{eq:error-estimate-4}
	\end{align}
Similarly, using Lemmas~\ref{lem:phi-sigma4-estimate} and \ref{lem:technical}, the relation $\e^{\phi,m+1} =\eA^{\phi,m+1}+\eh^{\phi,m+1}$, and a standard finite element error estimate, we arrive at
	\begin{align}
\varepsilon^{-1}\left|\iprd{\sigma_5^{\phi,m+\hf}}{\dtau \eh^{\phi,m+\hf}}\right| \le& \, C \norm{\nabla \sigma_5^{\phi,m+\hf}}{}^2 +  \frac{\alpha}{6}  \norm{\dtau \eh^{\phi,m+\hf}}{-1,h}^2
	\nonumber
	\\
\le& \, C \norm{\nabla \e^{\phi,m+1}}{}^2 + C \norm{\nabla \e^{\phi,m}}{}^2 + \frac{\alpha}{6}  \norm{\dtau \eh^{\phi,m+\hf}}{-1,h}^2
	\nonumber
	\\
\le& \, C \norm{\nabla \eA^{\phi,m+1}}{}^2 +C \norm{\nabla \eh^{\phi,m+1}}{}^2 + C \norm{\nabla \eA^{\phi,m}}{}^2 
	\nonumber
	\\
&+C \norm{\nabla \eh^{\phi,m}}{}^2 + \frac{\alpha}{6}  \norm{\dtau \eh^{\phi,m+\hf}}{-1,h}^2 
	\nonumber
	\\
\le& \, C h^{2q}\left| \phi^{m+1} \right|_{H^{q+1}}^2 + C \norm{\nabla \eh^{\phi,m+1}}{}^2 + C h^{2q}\left| \phi^{m} \right|_{H^{q+1}}^2 
	\nonumber
	\\
&+ C \norm{\nabla \eh^{\phi,m}}{}^2 + \frac{\alpha}{6}  \norm{\dtau \eh^{\phi,m+\hf}}{-1,h}^2.
	\label{eq:error-estimate-5}
	\end{align}
Applying Lemmas~\ref{lem:sigma6-estimate} and \ref{lem:technical}, the relation $\e^{\phi,m+1} =\eA^{\phi,m+1}+\eh^{\phi,m+1}$, and a standard finite element error estimate, we find that 
	\begin{align}
\varepsilon^{-1}\left|\iprd{\sigma_6^{\phi,m+\hf}}{\dtau \eh^{\phi,m+\hf}}\right| \le& \, C \norm{\nabla \sigma_6^{\phi,m+\hf}}{}^2 +  \frac{\alpha}{6}  \norm{\dtau \eh^{\phi,m+\hf}}{-1,h}^2
	\nonumber
	\\
\le & \, C \tau^3  \left( \int_{t_{m-1}}^{t_{m+1}}\norm{\nabla\partial_{ss}\phi(s)}{}^2 \, ds \right) + C\norm{\nabla \e_h^{\phi,m}}{}^2+C \norm{\nabla \e_h^{\phi,m-1}}{}^2 
	\nonumber
	\\
& + C h^{2q}\left| \phi^m \right|_{H^{q+1}}^2 + C  h^{2q}\left| \phi^{m-1} \right|_{H^{q+1}}^2 + \frac{\alpha}{6}  \norm{\dtau \eh^{\phi,m+\hf}}{-1,h}^2.
	\label{eq:error-estimate-6}
	\end{align}

The following inequality is a direct consequence of \eqref{eq:phi-truncation-5}: 
	\begin{equation}
\frac{\varepsilon\tau^2}{4} \, \aiprd{\ddtau\phi^m}{\dtau \eh^{\phi, m+\hf}}  \le  C  \frac{\tau^3}{3} \int_{t_{m-1}}^{t_m} \norm{\nabla \Delta \partial_{ss} \phi(s)}{}^2 ds + \frac{\alpha}{6} \norm{\dtau \eh^{\phi,m+\hf}}{-1,h}^2 .
	\label{eq:error-estimate-7}
	\end{equation}
	
Using Lemma \ref{lem:u-truncation-errors}, we also obtain		
	\begin{align}
\iprd{\sigma^{\bu,m+\hf}_1 + \sigma^{\bu,m+\hf}_2}{\bar{\e}_h^{\bu,m+\hf}} \le& \, C \norm{\sigma^{\bu,m+\hf}_1}{}^2 + C \norm{\sigma^{\bu,m+\hf}_2}{}^2 + \frac{\eta}{22 \gamma} \norm{\nabla \bar{\e}_h^{\bu,m+\hf}}{}^2 
	\nonumber
	\\
\le& \, \frac{\eta}{22 \gamma} \norm{\nabla \bar{\e}_h^{\bu,m+\hf}}{}^2 + C\frac{h^{2q+2}}{\tau} \int_{t_m}^{t_{m+1}} \norm{\partial_s\bu(s)}{H^{q+1}}^2  ds 
	\nonumber
	\\
&+ C \, \frac{\tau^3}{640} \int_{t_m}^{t_{m+1}}\norm{\partial_{sss}\bu(s)}{}^2 ds .
	\label{eq:error-estimate-8}
	\end{align}

Next, using~\eqref{eq:p-truncation-3},
	\begin{align}
-\frac{1}{\gamma} \cform{\bar{\e}_h^{\bu,m+\hf} }{\sigma_3^{p,m+\hf}} &\le \frac{\eta}{22 \gamma} \norm{\nabla \bar{\e}_h^{\bu,m+\hf}}{}^2 + C \norm{ \sigma^{p,m+\hf}_3}{}^2 
	\nonumber
	\\
&\le \frac{\eta}{22 \gamma} \norm{\nabla \bar{\e}_h^{\bu,m+\hf}}{}^2 + \frac{\tau^3}{96} \int_{t_m}^{t_{m+1}} \norm{\partial_{ss} p(s)}{}^2 ds,
	\end{align}
	
Now let's consider the convection trilinear terms. Adding and subtracting the appropriate terms, for all $1 \le m \le M-1$,
	\begin{align}
&\hspace{-0.5in}\left| - \bform{\phi^{m+\hf}}{\bu^{m+\hf}}{\eh^{\mu,m+\hf}} + \bform{\phitil^{m+\hf}}{\bbuh^{m+\hf}}{\eh^{\mu,m+\hf}} \right.
	\nonumber
	\\
& \hspace{-0.5in} \left.+ \bform{\phi^{m+\hf}}{\bar{\e}_h^{\bu,m+\hf}}{\mu^{m+\hf}} - \bform{\phitil^{m+\hf}}{\bar{\e}_h^{\bu,m+\hf}}{\muh^{m+\hf}} \right| 
	\nonumber
	\\
\le & \ \left| \bform{\sigma_6^{\phi, m+\hf}}{\bu^{m+\hf}}{\eh^{\mu, m+\hf}- \overline{\e_h^{\mu,m+\hf}}} \right| + \left|\bform{\sigma_6^{\phi, m+\hf}}{\bar{\e}_h^{\bu,m+\hf}}{\mu^{m+\hf}} \right|
	\nonumber
	\\
&  + \left|\bform{\phitil^{m+\hf}}{\bar{\e}_a^{\bu,m+\hf} - \sigma_3^{\bu,m+\hf}}{\eh^{\mu,m+\hf}-\overline{\e_h^{\mu,m+\hf}}} \right| + \left| \bform{\phitil^{m+\hf}}{\bar{\e}_h^{\bu,m+\hf}}{\eA^{\mu,m+\hf}}\right|
	\nonumber
	\\
\le & \ \norm{\nabla \sigma_6^{\phi, m+\hf}}{} \norm{\bu^{m+\hf}}{L^4} \norm{\eh^{\mu, m+\hf}- \overline{\e_h^{\mu,m+\hf}}}{L^4} + \norm{ \nabla \sigma_6^{\phi,m+\hf}}{} \norm{\bar{\e}_h^{\bu,m+\hf}}{L^4} \norm{\mu^{m+\hf}}{L^4}
	\nonumber
	\\
& + \norm{\nabla \phitil^{m+\hf}}{} \norm{\bar{\e}_a^{\bu,m+\hf} - \sigma_3^{\bu,m+\hf}}{L^4} \norm{\eh^{\mu,m+\hf}- \overline{\e_h^{\mu,m+\hf}}}{L^4}
	\nonumber
	\\
& + \norm{\nabla \phitil^{m+\hf}}{} \norm{\bar{\e}_h^{\bu,m+\hf}}{L^4} \norm{\eA^{\mu,m+\hf}}{L^4}
	\nonumber
	\\
\le & \  \frac{\varepsilon}{8} \norm{\nabla \eh^{\mu,m+\hf}}{}^2 + \frac{\eta}{22 \gamma} \norm{\nabla \bar{\e}_h^{\bu,m+\hf}}{}^2 + C \norm{\nabla \eh^{\phi,m+1}}{}^2 + C \norm{\nabla \eh^{\phi,m}}{}^2 
	\nonumber
	\\
& +  h^{2q}\left| \phi^{m} \right|_{H^{q+1}}^2 + h^{2q}\left| \phi^{m-1} \right|_{H^{q+1}}^2  + h^{2q}\left| \bu^{m+1} \right|_{H^{q+1}}^2  +  h^{2q}\left| \bu^{m} \right|_{H^{q+1}}^2 + h^{2q}\left| \bu^{m-1} \right|_{H^{q+1}}^2
	\nonumber
	\\
&   + h^{2q}\left| \mu^{m+\hf} \right|_{H^{q+1}}^2 + C \tau^3 \int_{t_{m-1}}^{t_{m+1}} \norm{\nabla \partial_{ss} \phi(s)}{}^2 ds  + C \tau^3 \int_{t_m}^{t_{m+1}} \norm{\nabla \partial_{ss} \bu(s)}{}^2 ds,
	\label{eq:error-estimate-10}
	\end{align}
where we have used Lemmas \ref{lem-improved-a-priori-stabilities}, \ref{lem:phi-truncation-errors}, \ref{lem:sigma6-estimate}.

Additionally, after adding and subtracting the appropriate terms, for $1 \le m \le M-1$, we have
	\begin{align}
& \hspace{-0.5in}\frac{1}{\gamma} \left|-\Bform{\bu^{m+\hf}}{\bu^{m+\hf}}{\bar{\e}_h^{\bu,m+\hf}} + \Bform{\butil^{m+\hf}}{\bbuh^{m+\hf}}{\bar{\e}_h^{\bu,m+\hf}} \right| 
	\nonumber
	\\
= & \ \frac{1}{\gamma} \left|\Bform{\sigma_3^{\bu,m+\hf}}{\bu^{m+\hf}}{\bar{\e}_h^{\bu,m+\hf} } - \Bform{\frac{\tau^2}{2} \ddtau  \bu^m}{\bu^{m+\hf}}{\bar{\e}_h^{\bu,m+\hf}} \right.
	\nonumber
	\\
& +\Bform{\tilde{\bu}^{m+\hf}}{\sigma_3^{\bu,m+\hf}}{\bar{\e}_h^{\bu,m+\hf}}   - \Bform{\tilde{\bu}^{m+\hf}}{\bar{\e}_a^{\bu,m+\hf}}{\bar{\e}_h^{\bu,m+\hf}} 
	\nonumber
	\\
&   + \Bform{\tilde{\e}_a^{\bu,m+\hf}}{\bar{\e}_a^{\bu,m+\hf}}{\bar{\e}_h^{\bu,m+\hf}}  - \Bform{\tilde{\e}_a^{\bu,m+\hf}}{\bar{\bu}^{m+\hf}}{\bar{\e}_h^{\bu,m+\hf}} 
	\nonumber
	\\
&  + \Bform{\tilde{\e}_h^{\bu,m+\hf}}{\bar{\e}_a^{\bu,m+\hf}}{\bar{\e}_h^{\bu,m+\hf} } - \Bform{\tilde{\e}_h^{\bu,m+\hf}}{\bar{\bu}^{m+\hf}}{\bar{\e}_h^{\bu,m+\hf}} 
	\nonumber
	\\
&  \left. - \Bform{\butil^{m+\hf}}{\bar{\e}_h^{\bu,m+\hf}}{\bar{\e}_h^{\bu,m+\hf}} \right|.
	\label{eq:error-estimate-11-pre}
	\end{align}
This is the same basic decomposition considered in Baker's paper~\cite{baker76}. We immediately see that the last term vanishes by anti-symmetry in the last two terms of $B$: $\Bform{\butil^{m+\hf}}{\bar{\e}_h^{\bu,m+\hf}}{\bar{\e}_h^{\bu,m+\hf}}=0$. We examine the other eight terms individually. Using estimate \eqref{lem:technical-Bform-1} of Lemma \ref{lem:technical-Bform},
	\begin{align}
\frac{1}{\gamma}\left|\Bform{\sigma_3^{\bu,m+\hf}}{\bu^{m+\hf}}{\bar{\e}_h^{\bu,m+\hf}} \right|  \le & \ C \norm{\nabla \sigma_3^{\bu,m+\hf}}{} \norm{\nabla \bu^{m+\hf}}{} \norm{\nabla \bar{\e}_h^{\bu,m+\hf}}{} 
	\nonumber
	\\
\le & \ \frac{\eta}{22 \gamma} \norm{\nabla \bar{\e}_h^{\bu,m+\hf}}{}^2  +  C \frac{\tau^3}{96} \int_{t_m}^{t_{m+1}} \norm{\nabla \partial_{ss} \bu(s)}{}^2 \, ds ;
	\label{eq:error-estimate-11.1}
		\end{align} 
		\begin{align}
\frac{1}{\gamma}\left|\Bform{\hf \tau^2 \ddtau  \bu^m}{\bu^{m+\hf}}{\bar{\e}_h^{\bu,m+\hf}}\right| \le & \ C \norm{\nabla \tau^2 \ddtau  \bu^m}{} \norm{\nabla \bu^{m+\hf}}{} \norm{\nabla \bar{\e}_h^{\bu,m+\hf}}{} 
	\nonumber
	\\
 \le & \ \frac{\eta}{22 \gamma} \norm{\nabla \bar{\e}_h^{\bu,m+\hf}}{}^2 +  \frac{\tau^3}{3} \int_{t_{m-1}}^{t_{m+1}} \norm{\nabla \partial_{ss} \phi(s)}{}^2 ds ;
	\label{eq:error-estimate-11.2}
	\end{align}
	\begin{align}
\frac{1}{\gamma}\left| \Bform{\tilde{\bu}^{m+\hf} }{\sigma_3^{\bu,m+\hf}}{\bar{\e}_h^{\bu,m+\hf} }\right| \le &  C \norm{\nabla \tilde{\bu}^{m+\hf}}{} \norm{\nabla \sigma_3^{\bu,m+\hf}}{} \norm{\nabla \bar{\e}_h^{\bu,m+\hf}}{} 
	\nonumber
	\\
\le &  \frac{\eta}{22 \gamma} \norm{\nabla \bar{\e}_h^{\bu,m+\hf}}{}^2 + C \norm{\nabla \sigma_3^{\bu,m+\hf}}{} 
	\nonumber
	\\
 \le &  \frac{\eta}{22 \gamma} \norm{\nabla \bar{\e}_h^{\bu,m+\hf}}{}^2 +  C \frac{\tau^3}{96} \int_{t_m}^{t_{m+1}} \norm{\nabla \partial_{ss} \bu(s)}{}^2 \, ds ;
	\label{eq:error-estimate-11.3}
	\end{align}
	\begin{align}
\frac{1}{\gamma}\left| \Bform{\tilde{\bu}^{m+\hf} }{\bar{\e}_a^{\bu,m+\hf}}{\bar{\e}_h^{\bu,m+\hf} } \right|  \le &  C  \norm{\nabla \tilde{\bu}^{m+\hf}}{} \norm{\nabla \bar{\e}_a^{\bu,m+\hf}}{} \norm{\nabla \bar{\e}_h^{\bu,m+\hf}}{}
	\nonumber
	\\
\le & \ \frac{\eta}{22 \gamma} \norm{\nabla \bar{\e}_h^{\bu,m+\hf}}{}^2 + C h^{2q} \left(\left| \bu^{m+1} \right|^2_{H^{q+1}} + \left| \bu^{m} \right|^2_{H^{q+1}}\right)
	\nonumber
	\\
& + C h^{2q} \left(\left| p^{m+1} \right|^2_{H^{q}} + \left| p^{m} \right|^2_{H^{q}} \right);
	\label{eq:error-estimate-11.4}
	\end{align}
	\begin{align}
\frac{1}{\gamma}\left| \Bform{\tilde{\e}_a^{\bu,m+\hf}}{\bar{\bu}^{m+\hf}}{\bar{\e}_h^{\bu,m+\hf}} \right| \le & \ \norm{\nabla\tilde{\e}_a^{\bu,m+\hf}}{} \norm{\nabla \bar{\bu}^{m+\hf}}{} \norm{\nabla \bar{\e}_h^{\bu,m+\hf}}{} 
	\nonumber
	\\
\le & \ \frac{\eta}{22 \gamma} \norm{\nabla \bar{\e}_h^{\bu,m+\hf}}{}^2 + C h^{2q} \left(\left| \bu^{m} \right|^2_{H^{q+1}} + \left| \bu^{m-1} \right|^2_{H^{q+1}}\right)
	\nonumber
	\\
& + C h^{2q} \left( \left| p^{m} \right|^2_{H^{q}} + \left| p^{m-1} \right|^2_{H^{q}} \right).
	\label{eq:error-estimate-11.6}
	\end{align}
Using the stability estimate~\eqref{eq:max-stability-velecity-approx-error},
	\begin{align}
\frac{1}{\gamma}\left| \Bform{\tilde{\e}_a^{\bu,m+\hf}}{\bar{\e}_a^{\bu,m+\hf}}{\bar{\e}_h^{\bu,m+\hf}} \right| \le & \ \norm{\tilde{\e}_a^{\bu,m+\hf}}{L^\infty} \norm{\nabla \bar{\e}_a^{\bu,m+\hf}}{} \norm{\nabla \bar{\e}_h^{\bu,m+\hf}}{} 
	\nonumber
	\\
 \le & \ C \norm{\nabla \bar{\e}_a^{\bu,m+\hf}}{} \norm{\nabla \bar{\e}_h^{\bu,m+\hf}}{} 
	\nonumber
	\\
\le & \ \frac{\eta}{22 \gamma} \norm{\nabla \bar{\e}_h^{\bu,m+\hf}}{}^2 + C h^{2q} \left(\left| \bu^{m+1} \right|^2_{H^{q+1}} + \left| \bu^{m} \right|^2_{H^{q+1}}\right)
	\nonumber
	\\
& + C h^{2q} \left(\left| p^{m+1} \right|^2_{H^{q}} + \left| p^{m} \right|^2_{H^{q}} \right);
	\label{eq:error-estimate-11.5}
	\end{align}
and, with the estimate~\eqref{lem:technical-Bform-5}, the inverse-Sobolev  inequality~\eqref{sobolev-inverse-estimate}, the Poincar\'{e} inequality, and  estimate~\eqref{eq:max-stability-velecity-approx-error} again, 
	\begin{align}
\frac{1}{\gamma}\left| \Bform{\tilde{\e}_h^{\bu,m+\hf}}{\bar{\e}_a^{\bu,m+\hf}}{\bar{\e}_h^{\bu,m+\hf} }  \right| \le & \  C \norm{\tilde{\e}_h^{\bu,m+\hf}}{} \norm{\nabla\bar{\e}_a^{\bu,m+\hf}}{} \norm{\bar{\e}_h^{\bu,m+\hf}}{L^\infty} 
	\nonumber
	\\
& + C \norm{\tilde{\e}_h^{\bu,m+\hf}}{} \norm{\nabla \bar{\e}_h^{\bu,m+\hf}}{} \norm{\bar{\e}_a^{\bu,m+\hf}}{L^\infty}
	\nonumber
	\\
\le & \ C \norm{\tilde{\e}_h^{\bu,m+\hf}}{} h^q (|\bu|_{H^{q+1}} + |p|_{H^{q}}) h^{\frac{1-d}{2}} \norm{\nabla \bar{\e}_h^{\bu,m+\hf}}{} 
	\nonumber
	\\
& + C \norm{\tilde{\e}_h^{\bu,m+\hf}}{} \norm{\nabla \bar{\e}_h^{\bu,m+\hf}}{} 
	\nonumber
	\\
\le & \ \frac{\eta}{22 \gamma} \norm{\nabla \bar{\e}_h^{\bu,m+\hf}}{}^2 + C \norm{\eh^{\bu,m}}{}^2 + C \norm{\eh^{\bu,m-1}}{}^2;
	\label{eq:error-estimate-11.7}
	\end{align}
	\begin{align}
\frac{1}{\gamma}\left| \Bform{\tilde{\e}_h^{\bu,m+\hf}}{\bar{\bu}^{m+\hf}}{\bar{\e}_h^{\bu,m+\hf}} \right| \le & \ \left(\norm{\bar{\bu}^{m+\hf}}{L^\infty} + \norm{\nabla\bar{\bu}^{m+\hf}}{L^4} \right)\norm{\tilde{\e}_h^{\bu,m+\hf}}{}  \norm{\nabla \bar{\e}_h^{\bu,m+\hf}}{} 
	\nonumber
	\\
\le & \ \frac{\eta}{22 \gamma} \norm{\nabla \bar{\e}_h^{\bu,m+\hf}}{}^2 + C \norm{\eh^{\bu,m}}{}^2 + C \norm{\eh^{\bu,m-1}}{}^2.
	\label{eq:error-estimate-11.8}
	\end{align}

Combining the estimates \eqref{eq:error-estimate-1} -- \eqref{eq:error-estimate-11.8} with the error equation \eqref{eq:error-eq}, the result follows.
	\end{proof}

	\begin{lem}
	\label{lem:dtau-phi-error-estimate}
Suppose that $(\phi, \mu, \bu, p)$ is a weak solution to \eqref{eq:weak-error-a} -- \eqref{eq:weak-error-b}, with the additional regularities in Assumption~\ref{asmp:initial-stability-3}.  Then, for any $h, \tau >0$, there exists a constant $C>0$, independent of $h$ and $\tau$, such that, for $1\le m\le M-1$,
	\begin{align}
\norm{\dtau \eh^{\phi,m+\hf}}{-1,h}^2 \le&\, 2 \,  \varepsilon^2 \norm{\nabla \eh^{\mu,m+\hf}}{}^2 +C \norm{\nabla \eh^{\phi,m}}{}^2 + C \norm{\nabla \eh^{\phi,m-1}}{}^2 
	\nonumber
	\\
&+ 5 \, C_2^2 \norm{\nabla \bar{\e}_h^{\bu,m+\hf}}{}^2 + C\mathcal{R}^{m+\hf},
 	\label{eq:dtau-phi-error-estimate}
	\end{align}
where $C_2 = C_0^2 C_1$, $C_0$ is the $H^1(\Omega) \hookrightarrow L^4(\Omega)$ Sobolev embedding constant, $C_1$ is a bound for $\max\limits_{0\le t \le T} \norm{\nabla \phih^m}{}$, and $\mathcal{R}^{m+\hf}$ is the consistency term given in \eqref{eq:consistency}.
	\end{lem}
	
	\begin{proof}
Define $\mathsf{T}_h: \Soh \rightarrow \Soh$ via the variational problem: given $\zeta \in \Soh$, find $\xi \in \Soh$ such that $\aiprd{\mathsf{T}_h (\zeta)}{\xi} = \iprd{\zeta}{\xi}$ for all $\xi \in \Soh$. Then, setting $\nu = \mathsf{T}_h\left(\dtau \eh^{\phi,m+\frac12} \right)$ in \eqref{eq:error-a} and combining, we have
	\begin{align}
\norm{\dtau \eh^{\phi,m+\hf}}{-1,h}^2  =& \, - \varepsilon\, \aiprd{\eh^{\mu,m+\hf}}{\mathsf{T}_h\left(\dtau \eh^{\phi,m+\hf} \right)} + \iprd{\sigma^{m+\hf}_1 + \sigma^{m+\hf}_2}{\mathsf{T}_h\left(\dtau \eh^{\phi,m+\hf} \right)} 
	\nonumber
	\\
&- \bform{\phi^{m+\hf}}{\bu^{m+\hf}}{\mathsf{T}_h\left(\dtau \eh^{\phi,m+\hf} \right)} + \bform{\phitil^{m+\hf}}{\bbuh^{m+\hf}}{\mathsf{T}_h\left(\dtau \eh^{\phi,m+\hf} \right)} 
	\nonumber
	\\
=& \, - \varepsilon\, \aiprd{\eh^{\mu,m+\hf}}{\mathsf{T}_h\left(\dtau \eh^{\phi,m+\hf} \right)} + \iprd{\sigma^{m+\hf}_1 + \sigma^{m+\hf}_2}{\mathsf{T}_h\left(\dtau \eh^{\phi,m+\hf} \right)} 
	\nonumber
	\\
&- \bform{\sigma_6^{\phi,m+\hf}}{\bu^{m+\hf}}{\mathsf{T}_h\left(\dtau \eh^{\phi,m+\hf} \right)} 
	\nonumber
	\\
&- \bform{\phitil^{m+\hf}}{\bar{\e}^{\bu,m+\hf} - \sigma_3^{\bu,m+\hf}}{\mathsf{T}_h\left(\dtau \eh^{\phi,m+\hf} \right)} 
	\nonumber
	\\
\le& \, \varepsilon \norm{\nabla \eh^{\mu,m+\hf}}{} \norm{\dtau \eh^{\phi,m+\hf}}{-1,h} + \norm{\sigma^{m+\hf}_1 + \sigma^{m+\hf}_2}{} \norm{\mathsf{T}_h\left(\dtau \eh^{\phi,m+\hf} \right)}{}
	\nonumber
	\\
&+  \norm{\sigma_6^{\phi,m+\hf}}{} \norm{\bu^{m+\hf}}{L^4} \norm{\mathsf{T}_h\left(\dtau \eh^{\phi,m+\hf} \right)}{L^4} 
	\nonumber
	\\
&+ \norm{\phitil^{m+\hf}}{L^\infty} \norm{\bar{\e}^{\bu,m+\hf} - \sigma_3^{\bu,m+\hf}}{} \norm{\mathsf{T}_h\left(\dtau \eh^{\phi,m+\hf} \right)}{}
	\nonumber
	\\
\le& \, \varepsilon^2 \norm{\nabla \eh^{\mu,m+\hf}}{}^2 + \frac{1}{4} \norm{\dtau \eh^{\phi,m+\hf}}{-1,h}^2 + C \norm{\sigma^{m+\hf}_2 + \sigma^{m+\hf}_1}{}^2 
	\nonumber
	\\
&+ C \norm{\sigma^{m+\hf}_6}{} \norm{\nabla \mathsf{T}_h\left(\dtau \eh^{\phi,m+\hf} \right)}{} + C_2 \norm{\nabla \bar{\e}_h^{\bu,m+\hf}}{} \norm{\nabla \mathsf{T}_h\left(\dtau \eh^{\phi,m+\hf} \right)}{}
	\nonumber
	\\
&+ C \norm{\nabla \left(\bar{\e}_a^{\bu,m+\hf} - \sigma_3^{\bu,m+\hf}\right)}{} \norm{\nabla \mathsf{T}_h\left(\dtau \eh^{\phi,m+\hf} \right)}{}
	\nonumber
	\\
\le& \, \varepsilon^2 \norm{\nabla \eh^{\mu,m+\hf}}{}^2 + \hf \norm{\dtau \eh^{\phi,m+\hf}}{-1,h}^2 +C \norm{\nabla \eh^{\phi,m}}{}^2 + C \norm{\nabla \eh^{\phi,m-1}}{}^2 
	\nonumber
	\\
&+ \frac{5 C_2^2}{2} \norm{\nabla \bar{\e}_h^{\bu,m+\hf}}{}^2  + C\mathcal{R}^{m+\hf} ,
	\end{align}
for $1 \le m \le M-1$, where we have used Lemmas~\ref{lem:phi-truncation-errors}, \ref{lem:u-truncation-errors}, and \ref{lem:sigma6-estimate}. The result now follows.
	\end{proof}

	\begin{lem}
	\label{lem:error-rhs-eh-control}
Suppose that $(\phi, \mu, \bu, p)$ is a weak solution to \eqref{eq:weak-error-a} -- \eqref{eq:weak-error-d}, with the additional regularities described in Assumption~\ref{asmp:initial-stability-3}.  Then, for any $h$, $\tau >0$, there exists a constant $C>0$, independent of $h$ and $\tau$, but possibly dependent upon $T$, such that, for any $1\le m\le M-1$,
	\begin{align}
&\frac{\varepsilon}{2\tau} \left( \norm{\nabla\eh^{\phi,m+1}}{}^2 - \norm{\nabla \eh^{\phi,m}}{}^2\right) + \frac{1}{2 \tau \gamma} \left(\norm{ \eh^{\bu,m+1}}{}^2 - \norm{\eh^{\bu,m}}{}^2\right) + \frac{ \varepsilon\tau^2}{4} \, \aiprd{\ddtau\eh^{\phi,m}}{\dtau \eh^{\phi, m+\hf}} 
	\nonumber
	\\
&\qquad+ \frac{\varepsilon}{2} \norm{\nabla \eh^{\mu,m+\hf}}{}^2 + \frac{\eta}{4\gamma} \, \norm{ \nabla \bar{\e}_h^{\bu,m+\hf}}{}^2 \le \, C \norm{\nabla \eh^{\phi,m+1}}{}^2 + C \norm{\nabla \eh^{\phi,m}}{}^2  
	\nonumber
	\\
&\qquad +  C \norm{\nabla \eh^{\phi,m-1}}{}^2 + C \norm{\eh^{\bu,m}}{}^2 + C \norm{\eh^{\bu,m-1}}{}^2 + C \mathcal{R}^{m+\hf} .
	\end{align}
	\end{lem}
	\begin{proof}
This follows upon combining the last two lemmas and choosing $\alpha$ in \eqref{eq:right-hand-side-estimate} appropriately.
	\end{proof}

Using the last lemma, we are ready to show the main convergence result for our second-order splitting scheme.
	\begin{thm}
	\label{thm:soch-error-estimate}
Suppose $(\phi, \mu, \bu, p)$ is a weak solution to \eqref{eq:weak-error-a} -- \eqref{eq:weak-error-d}, with the additional regularities described in Assumption~\ref{asmp:initial-stability-3}.  Then, provided that $0<\tau <\tau_0$, with some $\tau_0$ sufficiently small,
	\begin{align}
\max_{1 \le m \le M-1} \left(\norm{\nabla\eh^{\phi,m+1}}{}^2 + \norm{\eh^{\bu,m+1}}{}^2\right) +&\, \tau\sum_{m=1}^{M-1} \left(\norm{\nabla\eh^{\mu,m+\hf}}{}^2 + \norm{\nabla\bar{\e}_h^{\bu,m+\hf}}{}^2\right) \le C(T)(\tau^4+h^{2q})  , 
	\end{align}
for some $C(T)>0$ that is independent of $\tau$ and $h$.
	\end{thm}

	\begin{proof}
Using Lemma~\ref{lem:error-rhs-eh-control}, we have
	\begin{align}
\frac{1}{2 \tau} & \,\left( \norm{\nabla \eh^{\phi,m+1}}{}^2 - \norm{\nabla \eh^{\phi,m}}{}^2 \right) + \frac{1}{2 \tau \gamma} \left(\norm{ \eh^{\bu,m+1}}{}^2 - \norm{\eh^{\bu,m}}{}^2\right) +  \frac14 \norm{\nabla\eh^{\mu,m+\hf}}{}^2 \hspace{0.6in}&
	\nonumber
	\\
&+ \frac{\eta}{4\gamma} \, \norm{\nabla\bar{\e}_h^{\bu,m+\hf}}{}^2 + \frac{1}{8 \tau} \left( \norm{\nabla \eh^{\phi,m+1} - \nabla \eh^{\phi,m}}{}^2 - \norm{\nabla \eh^{\phi,m} - \nabla \eh^{\phi,m-1}}{}^2 \right) 
  	\nonumber
  	\\
\le& \, C \norm{\nabla \eh^{\phi, m+1}}{}^2 + C \norm{\nabla \eh^{\phi, m}}{}^2  +  C \norm{\nabla \eh^{\phi, m-1}}{}^2 + C \norm{\eh^{\bu,m}}{}^2 + C \norm{\eh^{\bu,m-1}}{}^2 
	\nonumber
	\\
&+ C\mathcal{R}^{m+\hf}.
	\label{eq:soch-error-sum}
	\end{align}
Now, applying $\tau\sum_{m=1}^\ell$ to \eqref{eq:soch-error-sum}, and observing that $\eh^{\phi,m}\equiv 0$ and $\eh^{\bu,m}\equiv {\bf 0}$, for $m = 0,1$, leads to 
	\begin{align}
\norm{\nabla\eh^{\phi,\ell+1}}{}^2 +& \ \norm{\eh^{\bu,\ell+1}}{}^2 + \frac{\tau}{2} \sum_{m=1}^{\ell} \left(\norm{\nabla\eh^{\mu,m+\hf}}{}^2 + \norm{\nabla \bar{\e}_h^{\bu,m+\hf}}{}^2\right) 
	\nonumber
	\\
& \le  C_3\tau\sum_{m=1}^\ell \mathcal{R}^{m+\hf} + C_4\tau\sum_{m=2}^{\ell+1}\norm{\nabla\eh^{\phi,m}}{}^2 + C_5\tau\sum_{m=2}^{\ell}\norm{\eh^{\bu,m}}{}^2.
	\label{eq:soch-estimate-before-gronwall}
	\end{align}
If $0< \tau \le \tau_0:= \frac{1}{2C_4} < \frac{1}{C_4}$, since $1\le \frac{1}{1-C_4\tau} \le 2$, it follows from the last estimate that
	\begin{align}
\norm{\nabla\eh^{\phi,\ell+1}}{}^2 +& \norm{\eh^{\bu,\ell+1}}{}^2 + \frac{\tau}{2} \sum_{m=1}^{\ell} \left(\norm{\nabla\eh^{\mu,m+\hf}}{}^2 + \norm{\nabla \bar{\e}_h^{\bu,m+\hf}}{}^2\right)  
	\nonumber
	\\
  \le & \, \frac{C_3\tau}{1-C_4\tau}\sum_{m=1}^\ell \mathcal{R}^{m+\hf} + \frac{C_4\tau}{1-C_4\tau} \sum_{m=2}^{\ell}\norm{\nabla\eh^{\phi,m}}{}^2 + \frac{C_5\tau}{1-C_4\tau} \sum_{m=2}^{\ell}\norm{\nabla\eh^{\bu,m}}{}^2
	\nonumber
	\\
\le & \, 2C_3C_6(\tau^4+h^{2q}) + 2C_7 \tau \sum_{m=2}^{\ell}\left( \norm{\nabla\eh^{\phi,m}}{}^2 +  \norm{\nabla\eh^{\bu,m}}{}^2\right),
	\label{eq:soch-pre-gronwall}
	\end{align}
where we have used the fact that $\tau\sum_{m=1}^{M-1} \mathcal{R}^{m+\hf}\le C_6(\tau^4+h^{2q})$ and where $C_7 := \max\left(C_4,C_5\right)$. Appealing to the discrete Gronwall inequality \ref{lem-discrete-gronwall}, it follows that, for any 	\begin{equation}
\norm{\nabla\eh^{\phi,\ell+1}}{}^2 + \norm{\eh^{\bu,\ell+1}}{}^2 + \frac{\tau}{2} \sum_{m=1}^{\ell} \left(\norm{\nabla\eh^{\mu,m+\hf}}{}^2 + \norm{\nabla \bar{\e}_h^{\bu,m+\hf}}{}^2\right)\le 2C_3C_6(\tau^4+h^{2q})\exp(2C_7T),
	\end{equation}
for any $1\le \ell \le M-1$.

	\end{proof}
	
	\begin{rem}
From here it is straightforward to establish an optimal error estimate of the form
	\begin{align}
\max_{1\le m \le M-1} \left( \norm{\nabla\e^{\phi,m+1}}{}^2 + \norm{\e^{\bu,m+1}}{}^2 \right) + \tau\sum_{m=1}^{M-1} \left( \norm{\nabla\e^{\mu,m+\hf}}{}^2 + \norm{\nabla\bar{\e}^{\bu,m+\hf}}{}^2  \right) \le C(T)(\tau^4+h^{2q}) 
	\end{align}
using $\ephi = \eAphi + \ehphi$, \emph{et cetera}, the triangle inequality, and the standard spatial approximations. We omit the details for the sake of brevity.
	\end{rem}
	
	\section*{Acknowledgment}
This work is supported in part by the grants NSF DMS-1418689 (C.~Wang), NSFC 11271281 (C.~Wang), NSF DMS-1418692 (S.~Wise), NSF DMS-1008852 (X.~Wang), and NSF DMS-1312701 (X.~Wang). 
	
	\appendix
	
	\section{Some Discrete Gronwall Inequalities}
	\label{appen:A}
	
We will need the following discrete Gronwall inequality cited in ~\cite{heywood90,layton08}:
	\begin{lem}	
	\label{lem-discrete-gronwall}
Fix $T>0$. Let $M$ be a positive integer, and define $\tau \le \frac{T}{M}$. Suppose $\left\{a_m\right\}_{m=0}^M$, $\left\{b_m\right\}_{m=0}^M$ and $\left\{c_m\right\}_{m=0}^{M-1}$ are non-negative sequences such that $\tau\sum_{m=0}^{M-1} c_m \le C_1$, where $C_1$ is independent of $\tau$ and $M$.  Further suppose that, 
	\begin{equation}
a_\ell + \tau \sum_{m=0}^{\ell} b_m \le C_2 +\tau \sum_{m=0}^{\ell-1} a_m c_m ,  \quad \forall \, 1\le \ell \le  M,	\label{eq:gronwall-assumption}
	\end{equation}
where $C_2>0$ is a constant independent of $\tau$ and $M$.  Then, for all $\tau>0$, 
	\begin{equation}
a_\ell +\tau \sum_{m=0}^{\ell} b_m \le C_2 \exp \left(\tau\sum_{m=0}^{\ell-1}c_m \right) \le C_2\exp(C_1) , \quad \forall \, 1\le \ell \le  M . 
	\label{eq:gronwall-conclusion}
	\end{equation}
	\end{lem}
Note that the sum on the right-hand-side of \eqref{eq:gronwall-assumption} must be explicit.

In addition, the following more general discrete Gronwall inequality is needed in the stability analysis. 
	\begin{lem}	
	\label{lem-discrete-gronwall-2}
Fix $T>0$. Let $M$ be a positive integer, and define $\tau \le \frac{T}{M}$. Suppose $\left\{a_m\right\}_{m=0}^M$, $\left\{b_m\right\}_{m=0}^M$ and $\left\{c_m\right\}_{m=0}^{M-1}$ are non-negative sequences such that $\tau\sum_{m=0}^{M-1} c_m \le C_1$, where $C_1$ is independent of $\tau$ and $M$.  Suppose that, for all $\tau>0$ and for some constant $0 < \alpha < 1$, 
	\begin{equation}
a_\ell + \tau \sum_{m=0}^{\ell} b_m \le C_2 +\tau \sum_{m=0}^{\ell-1} c_m \sum_{j=0}^m \alpha^{m-j} a_j ,  \quad \forall \, 1\le \ell \le  M , 
	\label{eq:gronwall-assumption-2}
	\end{equation}
where $C_2>0$ is a constant independent of $\tau$ and $M$.  Then, for all $\tau>0$, 
	\begin{equation}
a_\ell +\tau \sum_{m=0}^{\ell} b_m \le ( C_2 + a_0 C_1) \exp\left(\frac{C_1}{1-\alpha}\right) ,  \quad \forall \, 1\le \ell \le  M . 
	\label{eq:gronwall-conclusion-2}
	\end{equation}
	\end{lem}
	
	\begin{proof} 
We set $A_\alpha := \frac{1}{1-\alpha} > 1$.  A careful application of induction, using (\ref{eq:gronwall-assumption-2}), yields the following inequality: 
	\begin{equation}
a_\ell +\tau \sum_{m=0}^{\ell} b_m \le \prod_{m=1}^\ell d_{\ell,m} , \quad \forall \, 1\le \ell \le  M,
	\end{equation}
where
 
   \begin{equation}
d_{\ell,m} = \left\{
	\begin{array}{ccc}
\prod_{k=0}^{m-1}( 1 + \tau \alpha^k c_m ) & \mbox{if} &  1\le m\le \ell-1
	\\
C_2 + a_0 \tau \sum_{k=0}^{\ell-1} c_k \alpha^k  & \mbox{if} & m =\ell
	\end{array}
\right. .
	\label{lem A 2-1}
	\end{equation}
Meanwhile, the following bound is available: 
	\begin{align} 
d_{\ell,m} = & \ ( 1 + \tau c_m )  ( 1 + \alpha \tau c_m ) \cdots ( 1 + \alpha^{m-1} \tau c_m ) 
	\nonumber
	\\	
\le & \ \exp (  \tau c_m )   \exp (  \alpha \tau c_k ) \cdots \exp (  \alpha^{m-1} \tau c_m )  
	\nonumber 
	\\
= & \  \exp \left(  \tau ( 1 + \alpha + \cdots + \alpha^{m-1} ) c_m \right) 
    \le \exp \left(  A_\alpha c_m \tau \right) ,  \quad \forall \, 1\le m < \ell-1 , 
    \label{lem A 2-2} 
	\end{align}
which in turn leads to 
	\begin{align} 
d_{\ell,1}  d_{\ell,2} \cdots  d_{\ell,\ell-1}  \le & \ {\rm e}^{A_\alpha c_1 \tau}  {\rm e}^{A_\alpha c_2 \tau} \cdots  {\rm e}^{A_\alpha c^{\ell-1} \tau}   
	\nonumber 
	\\
\le & \ \exp \left( A_\alpha \tau ( c_1 + c_2 + \cdots + c_{\ell-1} ) \right) \le \exp( A_\alpha C_1 ) . 
    \label{lem A 2-3} 
	\end{align}
On the other hand, we also have 
	\begin{align} 
d_{\ell,\ell} = & \ C_2 + a_0 \tau \left( c_0 + c_1 \alpha + \cdots + c_{\ell-1} \alpha^{\ell-1} \right)  
	\nonumber
	\\
\le & \  C_2 + a_0 \tau \left( c_0 + c_1 + \cdots + c^{\ell-1} \right) \le C_2 + a_0 C_1 .  
	\label{lem A 2-4} 
	\end{align}  
In turn, a substitution of (\ref{lem A 2-3}) and (\ref{lem A 2-4}) into (\ref{lem A 2-1}) results in (\ref{eq:gronwall-conclusion-2}), the desired estimate. The proof of Lemma~\ref{lem-discrete-gronwall-2} is complete. 
\end{proof}

\bibliographystyle{plain}
	\bibliography{SOCHNS}

\end{document}